\documentclass[a4paper,11pt]{amsart}

\usepackage[utf8]{inputenc}
\usepackage{lmodern}
\usepackage{url}
\usepackage{amsmath, amsthm, amssymb, amscd, accents, bm}
\usepackage{mathtools}
\usepackage{mdwlist}
\usepackage{paralist}
\usepackage{graphicx}
\usepackage{epstopdf}
\usepackage{datetime}
\usepackage{afterpage}
\usepackage[dvipsnames]{xcolor}
\usepackage{hyperref}
\usepackage[ruled,vlined]{algorithm2e}
\usepackage{caption}
\usepackage{subcaption}

\usepackage[sort,nocompress]{cite}

\usepackage[foot]{amsaddr}

\usepackage{mathtools}

\newcommand\rurl[1]{%
  \href{https://#1}{\nolinkurl{#1}}%
}

\newtheorem{definition}{Definition}[section]
\newtheorem{theorem}[definition]{Theorem}
\newtheorem{lemma}[definition]{Lemma}
\newtheorem{corollary}[definition]{Corollary}
\newtheorem{proposition}[definition]{Proposition}

\newtheorem{remark}[definition]{Remark}

\newtheorem*{theorem*}{Theorem}

\def\N{{\mathbb N}}
\def\Z{{\mathbb Z}}
\def\R{{\mathbb R}}
\def\T{{\mathbb T}}
\def\C{{\mathbb C}}
\def\Q{{\mathbb Q}}

\newcommand{\Rd}{{\R^d}}
\newcommand{\Rtd}{{\R^{2d}}}
\newcommand{\Cd}{{\C^d}}

\newcommand{\lspan}{{\mathrm{span} \,}}

\newcommand{\lt}{{L^2(\R)}}
\newcommand{\ltd}{{L^2(\R^d)}}

\newcommand{\bone}{{\boldsymbol 1}}

\newcommand{\re}{{\mathrm{Re} \,}}

\newcommand{\ift}{{\mathcal{F}^{-1}}}
\newcommand{\ft}{{\mathcal{F}}}

\makeatletter
\newcommand\thankssymb[1]{\textsuperscript{\@fnsymbol{#1}}}

\makeatletter
\def\@makefnmark{%
  \leavevmode
  \raise.9ex\hbox{\fontsize\sf@size\z@\normalfont\tiny\@thefnmark}}

\makeatletter
\def\bign#1{\mathclose{\hbox{$\left#1\vbox to8.5\p@{}\right.\n@space$}}\mathopen{}}

\usepackage{eqparbox}
\usepackage{makebox}
\newcommand\textoverset[3][]{\mathrel{\overset{\scriptsize{\eqmakebox[#1]{#2}}}{#3}}}

\DeclareMathOperator*{\esssup}{ess\,sup}

\begin{document}

\title[Non-uniqueness theory in sampled STFT phase retrieval]{Non-uniqueness theory in sampled \\ STFT phase retrieval}

\author[Philipp Grohs]{Philipp Grohs\thankssymb{1}\textsuperscript{,}\thankssymb{2}\textsuperscript{,}\thankssymb{3}}
\address{\thankssymb{1}Faculty of Mathematics, University of Vienna, Oskar-Morgenstern-Platz 1, 1090 Vienna, Austria}
\address{\thankssymb{2}Research Network DataScience@UniVie, University of Vienna, Kolingasse 14-16, 1090 Vienna, Austria}
\address{\thankssymb{3}Johann Radon Institute of Applied and Computational Mathematics, Austrian Academy of Sciences, Altenbergstrasse 69, 4040 Linz, Austria}
\email{philipp.grohs@univie.ac.at}
\author[Lukas Liehr]{Lukas Liehr\thankssymb{1}}
\email{lukas.liehr@univie.ac.at}

\date{\today}
\maketitle

\begin{abstract}
The reconstruction of a function from its spectrogram (i.e., the absolute value of its short-time Fourier transform (STFT)) arises as a key problem in several important applications, including coherent diffraction imaging and audio processing.
It is a classical result that for suitable windows any function can, in principle, be uniquely recovered up to a global phase factor from its spectrogram.  However, for most practical applications only discrete samples -- typically from a lattice -- of the spectrogram are available. This raises the question of whether lattice samples of the spectrogram contain sufficient information for determining a function $f\in L^2(\mathbb{R}^d)$ up to a global phase factor.
In the present paper, we answer this question in the negative by providing general non-identifiability results which lead to a non-uniqueness theory for the sampled STFT phase retrieval problem.
Precisely, given any dimension $d$, any window function $g$ and any (symplectic or separable) lattice $\mathcal{L} \subseteq \mathbb{R}^d$, we construct pairs of functions $f,h\in L^2(\mathbb{R}^d)$ that do not agree up to a global phase factor, but whose spectrograms agree on $\mathcal{L}$. Our techniques are sufficiently flexible to produce counterexamples to unique recoverability under even more stringent assumptions; for example, if the window function is real-valued, the functions $f,h$ can even be chosen to satisfy $|f|=|h|$. Our results thus reveal the non-existence of a critical sampling density in the absence of phase information, a property which is in stark contrast to uniqueness results in time-frequency analysis.

\vspace{2em}

\end{abstract}

\break

\section{Introduction}

Given a square-integrable window function $g \in \ltd$ and a set $\mathcal{L} \subseteq \Rd \times \Rd$, a central problem in time-frequency analysis is the determination of the spanning properties of the Gabor system $\mathcal{G}(g,\mathcal{L}),$
$$
\mathcal{G}(g,\mathcal{L}) = \left \{ e^{2\pi i \lambda' \cdot t} g(t-\lambda) : (\lambda,\lambda') \in \mathcal{L} \right \} \subseteq \ltd .
$$
The completeness of the system $\mathcal{G}(g,\mathcal{L})$ in $\ltd$ is equivalent to the validity of the implication
$$
\langle f, e^{2 \pi i \lambda' \cdot} g(\cdot - \lambda) \rangle_{\ltd} = 0 \ \forall (\lambda,\lambda') \in \mathcal{L} \implies f = 0
$$
whenever $f \in \ltd$. Alternatively, this means that $f$ is uniquely determined by samples on $\mathcal{L}$ of its short-time Fourier transform (STFT) with respect to the window function $g$, which is the map
$$
V_gf(x,\omega) = \langle f, e^{2 \pi i \omega \cdot} g(\cdot - x) \rangle_{\ltd} = \int_\Rd f(t) \overline{g(t-x)} e^{-2\pi i \omega \cdot t} \, dt.
$$
The recovery of $f$ by samples of $V_gf$, in particular by samples located on a lattice, has been the subject of intensive research and produced a vast literature; possible starting points are \cite{Groechenig, Heil2007}. Classical results in this field state that under some mild conditions on the window function $g$ and a suitable density assumption on $\mathcal{L}$, the Gabor system $\mathcal{G}(g,\mathcal{L})$ forms a frame for $\ltd$ and hence every $f \in \ltd$ is uniquely and stably determined by the STFT-samples $V_gf(\mathcal{L}) \coloneqq \{ V_gf(z) : z \in \mathcal{L} \}$. Suppose now that we have incomplete information about the STFT, namely only samples of its modulus at $\mathcal{L}$, represented by the set
$$
|V_gf(\mathcal{L})| \coloneqq \{ |V_gf(z)| : z \in \mathcal{L} \}.
$$
The function $|V_gf|$ is called the spectrogram of $f$ with respect to $g$. The problem of trying to invert the map which sends $f$ to the magnitude-only measurements $|V_gf(\mathcal{L})|$,
$
f \mapsto |V_gf(\mathcal{L})|
$, is known as the \emph{STFT phase retrieval} problem. This non-linear inverse problem is gaining rapidly growing interest due to its occurrence in a remarkably wide number of applications, most notably in the areas of coherent diffraction imaging, radar, audio processing, and quantum mechanics \cite{appl1,appl2,appl3,appl4,appl5,appl6,appl7,appl8}. The STFT phase retrieval problem constitutes a challenging mathematical problem; a crucial reason for this lies in the appearance of severe instabilities \cite{daubcahill,alaifariGrohs,ALAIFARI2021401}. 

Observe that two functions $f,h \in \ltd$ which \emph{agree up to a global phase}, i.e., there exists a unimodular constant $\nu \in \T \coloneqq \{ z \in \C : |z|=1 \}$ such that $f=\nu h$, are indistinguishable by their spectrograms since in this situation one has $|V_gf(\mathcal{L})| = |V_gh(\mathcal{L})|$. The uniqueness problem in STFT phase retrieval asks for assumptions on $\mathcal{L} \subseteq \Rtd$ and $g \in \ltd$ such that every $f \in \ltd$ is uniquely determined up to the ambiguity of a global phase factor by $|V_gf(\mathcal{L})|$. Concisely, one seeks to find conditions on $g$ and $\mathcal{L}$ so that the implication
\begin{equation}\label{eq:uniqueness_implication}
    |V_gf(z)| = |V_gh(z)| \ \forall z \in \mathcal{L} \implies \exists \nu \in \T : f=\nu h
\end{equation}
holds true whenever $f,h \in \ltd$. This problem is well-understood if $\mathcal{L}$ is a continuous domain, such as an open set. However, in all practical applications the set $\mathcal{L}$ is a discrete set of samples in $\Rtd$. Deriving uniqueness results from discrete sampling sets is a key step towards both discretizing the infinite-dimensional STFT phase problem and making the problem computationally tractable. The discretization problem has recently attracted attention in the mathematical community \cite{grohsLiehrJFAA, grohsliehr1, grohsliehr2, alaifari2020phase, att2,grra}. However, the uniqueness problem from measurements arising from discrete sampling sets remains much less well-understood than its continuous counterpart. In view of practical applications, the most important discrete sampling sets are lattices, i.e., $\mathcal{L} = L\Z^{2d}$ for some invertible matrix $L \in \mathrm{GL}_{2d}(\R)$. The natural question arises of whether similarly to the case where phase information is present, every map $f \in \ltd$ is determined (up to a global phase) by $|V_gf(\mathcal{L})|$ assuming that the lattice $\mathcal{L}$ satisfies a density condition and $g$ has certain regularity properties. 
For the univariate case $d=1$ we highlighted in our previous work \cite{grohsLiehrJFAA} the following fundamental difference of the uniqueness question in the absence of phase information: there exists \emph{no} window function $g$ and \emph{no} lattice $\mathcal{L}$ such that the implication \eqref{eq:uniqueness_implication} holds true for every $f,h \in \ltd$. 

Following this line of research, the present article provides a systematic study of fundamental discretization barriers arising in the STFT phase retrieval problem. Since the univariate case $d=1$ is not sufficient for a majority of applications, we aim to establish a unified approach to the derivation of non-discretizability statements in STFT phase retrieval, independent of the dimension $d$. The derived theorems show general discretization barriers in arbitrary dimensions and various settings, such as in induced Pauli-type problems and real-valued signal regimes. We further reveal the non-existence of a critical sampling density in the absence of phase information, a property which is in stark contrast to uniqueness results in time-frequency analysis. For the precise mathematical statements of our results, we refer the reader to section \ref{sec:contribution}. To give further motivation for our paper, we now explain an important application where the problem of discretization appears in a natural way.

\subsection{Motivation: Ptychographic imaging}\label{sec:ptychography}

Ptychography is a highly successful method in the field of diffraction imaging. It has gained a great deal of attention in recent years since it provides a methodology which leads to the detailed understanding of the structure and properties of materials \cite{appl3,appl4, appl5}. The principle of ptychographic imaging is visualized in Figure \ref{fig:pt} and consists of the following three steps (indicated by (1), (2), and (3) in the figure):
\begin{enumerate}
    \item A so-called probe, which is a pinhole with a certain shape and which can be moved in the $x_1$ or $x_2$-direction, is modelled as the $(x_1,x_2)$-shift of a function $g \in L^2(\R^2)$.
    \item The probe is set in front of an unknown object which is modeled as a function $f$ of finite energy, $f \in L^2(\R^2)$. An electron or X-ray beam enters the probe (resp., pinhole) which leads to a concentration and localization of the beam. Subsequently, the beam hits a small part of the object.
    \item At a certain distance from the unknown object, the diffracted beam becomes the Fourier transform of the product of $f$ and the $(x_1,x_2)$-shift of the probe $g$, i.e., $\ft(fT_{(x_1,x_2)}g)$. A digital device, such as a camera, measures the intensities of this diffracted wave, resulting in a diffraction pattern. The measured intensities are the absolute value of $\ft(fT_{(x_1,x_2)}g)$, in other words, $|V_gf((x_1,x_2), \omega)|$ with $\omega \in \mathcal{B} \subseteq \R^2$. Due to the usage of a digital measuring device, the set $\mathcal{B}$ is a subset of a lattice in $\R^2$.
\end{enumerate}
For each position $(x_1,x_2) \in \R^2$ of the probe, one obtains a set of measurements of the form $$\{ |V_gf((x_1,x_2),\omega)| : \omega \in \mathcal{B} \}.$$ Since one performs a so-called raster-scan of the object (the scanning of the object with overlapping windows), the positions $(x_1,x_2)$ of the probe lie on a second lattice $\mathcal{A} \subseteq \R^2$. The complete data set is therefore a subset of
\begin{equation}\label{dataset}
    \{ |V_gf(z)| : z \in \mathcal{A} \times \mathcal{B} \},
\end{equation}
i.e., the samples are taken on a subset of the lattice $\mathcal{A} \times \mathcal{B} \subseteq \R^4$. The final goal is to reconstruct $f \in \lt$ from the data set \eqref{dataset}. The natural question arises of whether the data set \eqref{dataset} determines the object $f$ uniquely (up to a global phase). The theory developed in the present article shows that this is not the case, independent of the particular choices of $\mathcal{A}$ and $\mathcal{B}$. Since $\mathcal{A} \times \mathcal{B}$ is a separable lattice, i.e., the Cartesian product of two lattices, we pay particular attention to lattices of this structure.

\begin{figure}
\centering
\includegraphics[width=1\linewidth]{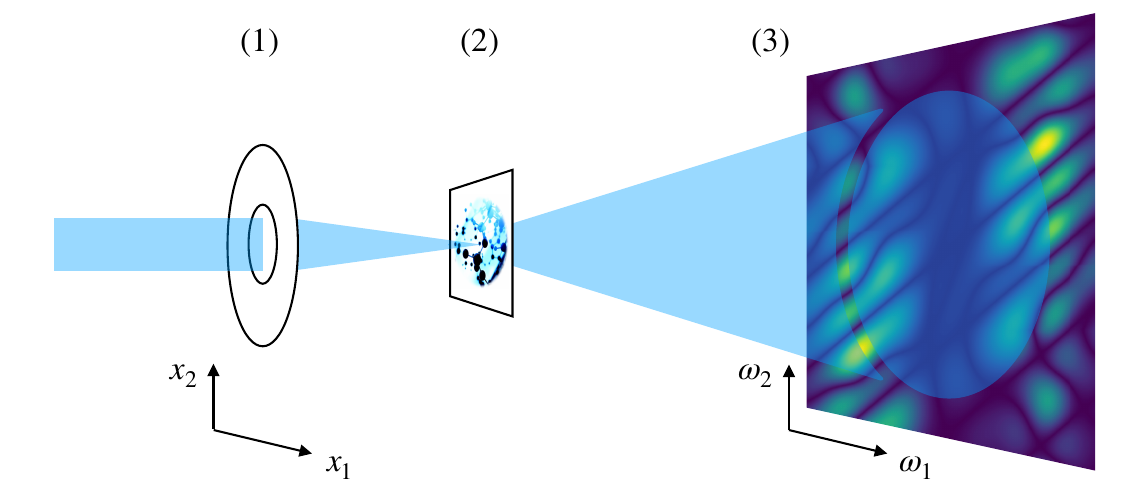}
\caption{The principle of Ptychographic imaging.}
\label{fig:pt}
\end{figure}

\subsection{Contributions}\label{sec:contribution}

In this section we shall provide a description of our main results. Recall that a set $\mathcal{L} \subseteq \R^{2d}$ is called a lattice if there exists an invertible matrix $L \in \mathrm{GL}_{2d}(\R)$ such that $\mathcal{L} = L \Z^{2d} = \{ Lz : z \in \Z^{2d} \}$. The matrix $L$ is called a generating matrix of $\mathcal{L}$. A matrix $S \in \mathrm{GL}_{2d}(\R)$ is called symplectic if $S^T \mathcal{J}S = \mathcal{J}$ where $\mathcal{J}$ denotes the standard symplectic matrix
$$
\mathcal{J} \coloneqq \begin{pmatrix} 0 & -I_d \\ I_d & 0 \end{pmatrix},
$$
$I_d$ is the identity matrix in $\R^{d \times d}$ and $0 \coloneqq 0_d$ denotes the zero matrix in $\R^{d \times d}$. The collections of all symplectic matrices in $\R^{2d \times 2d}$ is called the (real) symplectic group and is denoted by $\mathrm{Sp}_{2d}(\R)$. A lattice $\mathcal{L} = L\Z^{2d}$ is called symplectic if $L=\alpha S$ for some $\alpha >0$ and some $S \in \mathrm{Sp}_{2d}(\R)$, i.e., $\mathcal{L}$ is generated by a scaled symplectic matrix. We focus on studying the STFT phase retrieval problem in the situation where the samples arise from lattices of the form $ST\Z^{2d}$, where $S$ is a symplectic matrix and $T$ is block-diagonal. This has three reasons. Firstly, it allows us to apply tools from symplectic geometry, in particular, metaplectic operators which will constitute a crucial methodology in the construction of counterexamples. Secondly, this covers separable lattices which arise naturally in a variety of applications, such as ptychographic imaging (see the explanations in Section \ref{sec:ptychography}). Finally, uniqueness results in time-frequency analysis are phrased in terms of sampling on symplectic lattices, and we aim to develop our theory in a parallel fashion which allows for a direct comparison to classical sampling results in time-frequency analysis.

In the present exposition we work with the following equivalence relation which indicates that two functions $f_1,f_2$ agree up to a global phase factor:
$$
f_1 \sim f_2 \iff \exists \nu \in \T : f_1 = \nu f_2.
$$
If $f_1$ and $f_2$ do not agree up to a global phase, then we write $f_1 \nsim f_2$.

\begin{theorem}\label{thm:start1}
Suppose that $\mathcal{L} \subseteq \R^{2d}$ satisfies one of the following two conditions:
\begin{enumerate}
    \item $\mathcal{L} = S(\R^d \times \Lambda) = \{ S(x,\lambda) : x \in \R^d, \lambda \in \Lambda \}$ with $S\in \mathrm{Sp}_{2d}(\R)$ a symplectic matrix and $\Lambda \subseteq \R^d$ an arbitrary lattice.
    \item $\mathcal{L} = L\Z^{2d}$ is a lattice generated by $L \in \mathrm{GL}_{2d}(\R)$ which satisfies the factorization property $L=ST$ with $S\in \mathrm{Sp}_{2d}(\R)$ symplectic and $T$ a block-diagonal matrix of the form 
    $$
T = \begin{pmatrix} A & 0 \\ 0 & B \end{pmatrix}, \ \ A,B \in \mathrm{GL}_d(\R).
$$
\end{enumerate}
Then for every window function $g \in \ltd$ there exist two functions $f_1,f_2 \in \ltd$ such that
$$
|V_g(f_1)(z)| = |V_g(f_2)(z)| \ \forall z \in \mathcal{L} \ \ \text{and} \ \ f_1 \nsim f_2.
$$
\end{theorem}

Crucial consequences of this theorem are nonuniqueness results for the following class of important lattices which appear throughout time-frequency analysis and its applications.

\begin{corollary}\label{thm:start2}
Suppose that $\mathcal{L} \subseteq \R^{2d}$ is a lattice which satisfies one of the following four conditions:
\begin{enumerate}
    \item $\mathcal{L}$ is symplectic, i.e., $\mathcal{L}=\alpha S \Z^{2d}$ for some $\alpha>0$ and some $S \in \mathrm{Sp}_{2d}(\R)$.
    \item $\mathcal{L}$ is rectangular, i.e., its generating matrix is diagonal.
    \item $\mathcal{L}$ is separable, i.e., $\mathcal{L} = \mathcal{A} \times \mathcal{B}$ with $\mathcal{A},\mathcal{B}$ lattices in $\Rd$.
    \item $\mathcal{L}=\alpha \Z^d \times \beta \Z^d$ for some $\alpha,\beta \in \R \setminus \{ 0 \}$.
\end{enumerate}
Then for every window function $g \in \ltd$ there exists two functions $f_1,f_2 \in \ltd$ such that
$$
|V_g(f_1)(z)| = |V_g(f_2)(z)| \ \forall z \in \mathcal{L} \ \ \text{and} \ \ f_1 \nsim f_2.
$$
\end{corollary}

Notice that classical results in time-frequency analysis show that under some mild conditions on a window function $g \in \ltd$ and a certain density assumption on a symplectic or separable lattice $\mathcal{L} \subseteq \Rtd$, every $f \in \ltd$ is determined uniquely by the STFT samples $V_gf(\mathcal{L})$. The conclusions made in Corollary \ref{thm:start2} are in stark contrast to the uniqueness results where phase information is present. If a lattice $\mathcal{L}$ satisfies one of the properties of Corollary \ref{thm:start2}, then in the absence of phase information there exists \emph{no} window function $g \in \ltd$ such that every $f \in \ltd$ is determined (up to a global phase) by $|V_gf(\mathcal{L})|$ and this result is independent of the density of $\mathcal{L}$ (the density of $\mathcal{L}$ refers to the quantity $|\det L|^{-1}$ where $\mathcal{L}=L\Z^{2d}$). We refer to Section \ref{sec:implications} for a detailed discussion and comparison with the setting where phase information is present (resp., absent).

The conclusions of Theorem \ref{thm:start1} and Corollary \ref{thm:start2} show the existence of a function pair $(f_1,f_2)$ with the properties that $f_1 \nsim f_2$ and $f_1$ and $f_2$ produce identical phaseless STFT samples on certain subsets of $\mathcal{L} \subseteq \Rtd$. We now investigate the size of the class of functions which have these properties. To that end, we define for a window function $g \in \ltd$ and a set $\mathcal{L} \subseteq \Rtd$ the function class $\mathcal{N}(g,\mathcal{L})$ by
\begin{equation*}
    \begin{split}
        & \mathcal{N}(g,\mathcal{L}) \\
        & \coloneqq \left \{ f \in \ltd : \exists h \in \ltd \ \text{s.t.} \ h \nsim f \ \text{and} \ |V_gf(\mathcal{L})| = |V_gh(\mathcal{L})| \right \}.
    \end{split}
\end{equation*}
Clearly, $\mathcal{N}(g,\mathcal{L}) = \emptyset$ if and only if every $f \in \ltd$ is uniquely determined up to a global phase by $|V_gf(\mathcal{L})|$. On the other hand, if $\mathcal{L}$ satisfies one of the assumptions of Theorem \ref{thm:start1} or Corollary \ref{thm:start2}, then for every window function $g \in \ltd$ we have $\mathcal{N}(g,\mathcal{L}) \neq \emptyset$ which is equivalent to the statement that uniqueness is not guaranteed in the corresponding phase retrieval problem with window function $g$ and sampling locations $\mathcal{L}$.

Recall that a set $C \subseteq V$ of a real or complex vector space $V$ is called a cone if
$$
\kappa C \subseteq C
$$
for every $\kappa > 0$. A cone is called infinite-dimensional if it is not contained in any finite-dimensional subspace of $V$.

\begin{theorem}\label{thm:complex_cone}
Let $g \in \ltd$ be an arbitrary window function. If $\mathcal{L} \subseteq \Rtd$ satisfies $\mathcal{L} \subseteq S(\Rd \times \Lambda)$ for some symplectic matrix $S \in \mathrm{Sp}_{2d}(\R)$ and some lattice $\Lambda \subseteq \Rd$ then $\mathcal{N}(g,\mathcal{L})$ contains an infinite-dimensional cone.
\end{theorem}

Next, we pay particular attention to the setting where the window function $g$ is assumed to be real-valued and the lattice $\mathcal{L}$ is separable, i.e. there exist lattices $\mathcal{A},\mathcal{B} \subseteq \Rd$ such that $\mathcal{L} = \mathcal{A} \times \mathcal{B}$. These are prototypical assumptions which are made in numerous applications; prime examples for a real-valued window functions are Gaussians and frequently appearing lattices are of the form $\mathcal{L} = \alpha \Z^d \times \beta \Z^d$ (often referred to as time-frequency lattices).

A first consequence drawn from the real-valuedness assumption on $g$ and the separability assumption on $\mathcal{L}$ is that the Pauli-problem induced by the STFT is not unique. The Pauli problem dates back to the seminal work by Pauli \cite{Pauli} and concerns the question of whether a function $f \in \lt$ is determined up to a global phase by its modulus $|f|$ and the modulus of its Fourier transform $|\ft f|$ \cite{Pauli}. It is well-known that this is not the case, i.e., one can construct so-called Pauli-partners $f_1,f_2$ for which $|f_1|=|f_2|$ and $|\ft f_1| = |\ft f_2|$ but $f_1 \nsim f_2$ \cite{corbett_hurst_1977, VOGT1978365}. A natural extension of the Pauli problem replaces the Fourier transform by a different unitary operator or a finite number of unitary operators (see, for instance, Jaming's work on uniqueness from fractional Fourier transform measurements \cite{Jaming}, corresponding non-uniqueness results by Carmeli et al. \cite{CARMELI}, or the very recent paper by Jaming and Rathmair on uniqueness from three unitary operators \cite{JamingRathmair}). In Corollary \ref{thm:start2} we showed that a function $f \in \ltd$ is not determined by its spectrogram samples given on a separable lattice, independent of the particular choice of the window function $g$. If the window function is real-valued, then in the same spirit as in the Pauli problem, this non-uniqueness property remains valid even under prior knowledge of the modulus of $f$, as we show next.

\begin{theorem}\label{thm:start3}
Let $g \in L^2(\Rd,\R)$ be a real-valued window function, and let $\mathcal{L} \subseteq \Rtd$ be a separable lattice. Then there exist functions $f_1,f_2 \in \ltd$ which satisfy all of the following three properties:
\begin{enumerate}
    \item $|V_g(f_1)(z)|=|V_g(f_2)(z)|$ for all $z \in \mathcal{L}$
    \item $|f_1(t)| = |f_2(t)|$ for all $t \in \Rd$
    \item $f_1 \nsim f_2$.
\end{enumerate}
\end{theorem}

Suppose now that we impose a restriction on the underlying signal space. Precisely, we assume that all input functions are real-valued, and we inquire about the problem of whether such signals are determined up to a sign-factor by spectrogram samples given on a separable lattice. In this case, the corresponding phase retrieval problem is also known as the sign retrieval problem. In recent years, the reconstruction of a real-valued signal belonging to certain shift-invariant spaces (e.g. Paley-Wiener spaces, Gaussian shift-invariant spaces, shift-invariant spaces with totally positive generator) by samples of its modulus was considered by several authors; see, for instance, \cite{1Thakur2011, 2Groechenig2020, 3Romero2021, 4Alaifari2017}.
In the context of STFT phase retrieval, the following theorem demonstrates that the sign retrieval problem fails to be unique.

\begin{theorem}\label{thm:intro:sign_retrieval}
Let $g \in L^2(\Rd,\R)$ be a real-valued window function, and suppose that $\mathcal{L} \subseteq \R^{2d}$ satisfies one of the following two conditions:
\begin{enumerate}
    \item $\mathcal{L} = \Lambda \times \R^d$ with $\Lambda$ a lattice in $\Rd$.
    \item $\mathcal{L}$ is a separable lattice.
\end{enumerate}
Then there exist two real-valued functions $f_1,f_2 \in L^2(\Rd,\R)$ such that
$$
|V_g(f_1)(z)| = |V_g(f_2)(z)| \ \forall z \in \mathcal{L} \ \ \text{and} \ \ f_1 \nsim f_2.
$$
\end{theorem}

In view of numerical applications, there is one consequence of Theorem \ref{thm:intro:sign_retrieval} which is sufficiently important to have the status of a separate statement.

\begin{corollary}\label{thm:start5}
Suppose that $\mathcal{L} = L\Z^{2d} \subseteq \R^{2d}$ is a lattice which has the property that the generating matrix $L$ has rational entries, $L \in \mathrm{GL}_{2d}(\Q)$. Then for every real-valued window function $g \in L^2(\Rd,\R)$ there exists two real-valued functions $f_1,f_2 \in L^2(\Rd,\R)$ such that
$$
|V_g(f_1)(z)| = |V_g(f_2)(z)| \ \forall z \in \mathcal{L} \ \ \text{and} \ \ f_1 \nsim f_2.
$$
\end{corollary}

In a similar fashion, we define for a window function $g \in \ltd$ and a subset $\mathcal{L} \subseteq \Rtd$ the set $\mathcal{N}_\R(g,\mathcal{L})$ via
\begin{equation*}
    \begin{split}
        & \mathcal{N}_\R(g,\mathcal{L})  \\ & \coloneqq \left \{ f \in L^2(\Rd, \R) :  \exists h \in L^2(\Rd, \R) \ \text{s.t.} \ h \nsim f \ \text{and} \ |V_gf(\mathcal{L})| = |V_gh(\mathcal{L})| \right \}.
    \end{split}
\end{equation*}
If $\mathcal{N}_\R(g,\mathcal{L}) = \emptyset$ then every $f \in L^2(\Rd,\R)$ is determined up to a sign factor by $|V_gf(\mathcal{L})|$ within the signal class $L^2(\Rd,\R)$. If $g \in L^2(\Rd,\R)$ is real-valued and $\mathcal{L} \subseteq \R^{2d}$ is of the form $\mathcal{L}=\Lambda \times \Rd$ then according to Theorem \ref{thm:intro:sign_retrieval} it holds that $\mathcal{N}_\R(g,\mathcal{L}) \neq \emptyset$. In fact, $\mathcal{N}_\R(g,\mathcal{L})$ contains an infinite-dimensional cone.

\begin{theorem}\label{thm:real_cone}
Let $g \in \ltd$ be an arbitrary real-valued window function. If $\mathcal{L} \subseteq \Rtd$ satisfies $\mathcal{L} \subseteq \Lambda \times \Rd$ for some lattice $\Lambda \subseteq \Rd$ then $\mathcal{N}_\R(g,\mathcal{L})$ contains an infinite-dimensional cone.
\end{theorem}

\subsection{Outline} The paper is structured as follow. Section \ref{sec:preliminaries} presents the necessary background on time-frequency analysis, symplectic matrices, and metaplectic operators as well as lattices. Further, this section introduces the concept of shift-invariant spaces induced by a metaplectic operators, a key concept which will be used throughout the article. The exposition continues with Section \ref{sec:main_results_3} where the main results of the paper are presented and proved. In Section \ref{sec:spec} we construct function pairs which have the property that their spectrograms agree on certain lattices and semi-discrete sampling sets. In Section \ref{sec:noneq} we characterize all constructed functions pairs, which do not agree up to a global phase. The resulting consequences for the STFT phase retrieval problem (in particular, Theorem \ref{thm:start1}, Corollary \ref{thm:start2}, and Theorem \ref{thm:complex_cone} above) are presented in Section \ref{sec:implications} and together with visualizations in Section \ref{sec:examples_v}. Section \ref{sec:separability} is devoted to the proofs of Theorem \ref{thm:start3}, Theorem \ref{thm:intro:sign_retrieval}, Corollary \ref{thm:start5}, and Theorem \ref{thm:real_cone}.

\section{Preliminaries}\label{sec:preliminaries}

In this section we discuss several definitions and basic results which are needed through the remainder of the article.

\subsection{Time-frequency analysis}

Let $\ltd$ be the Lebesgue space of all measurable, square-integrable functions $f : \Rd \to \C$. The subspace consisting of all real-valued functions in $\ltd$ is denoted by $L^2(\Rd,\R)$. The inner product $\langle a,b \rangle \coloneqq \int_\Rd a(t)\overline{b(t)} \, dt$ renders $\ltd$ into a Hilbert space with induced norm $\| \cdot \|_2$. For $\tau,\nu \in \Rd$, the translation, modulation, and reflection operator are denoted by $T_\tau : \ltd \to \ltd, M_\nu : \ltd \to \ltd$, and $\mathcal{R} : \ltd \to \ltd$, and are defined by
\begin{equation*}
    \begin{split}
        T_\tau h (t) &= h(t-\tau) \\
        M_\nu h(t) &= e^{2 \pi i \nu \cdot t} h(t) \\
        \mathcal{R} h(t) &= h(-t).
    \end{split}
\end{equation*}
Note that the quantity $\nu \cdot t \coloneqq \sum_{j=1}^d \nu_j t_j, t=(t_1, \dots, t_d), \nu=(\nu_1,\dots,\nu_d),$ denotes the Euclidean inner product of two vectors in $\Rd$. The short-time Fourier transform (STFT) of a function $f \in \ltd$ with respect to a window function $g \in \ltd$ is the map $V_gf : \Rtd \to \C$, defined as
$$
V_gf(x,\omega) \coloneqq \langle f,M_\omega T_x g \rangle = \int_\Rd f(t)\overline{g(t-x)} e^{-2\pi i \omega \cdot t} \, dt, \ \ x,\omega \in \Rd.
$$
Observe that $V_gf(x,\cdot) = \ft(f\overline{T_x g})$ where $\ft$ denotes the Fourier transform which is defined on $L^1(\Rd) \cap \ltd$ via
$$
\ft h(\omega) \coloneqq \hat h(\omega) \coloneqq \int_\Rd h(t)e^{-2\pi i \omega \cdot t} \, dt
$$
and extends to a unitary operator mapping $\ltd$ bijectively onto $\ltd$. The following lemma summarizes elementary properties of the operators introduced beforehand (see, for instance, \cite[Chapters 1--3]{Groechenig}).

\begin{lemma}\label{lma:stft_properties}
Let $\tau,\nu \in \Rd$ and let $f,g \in \ltd$. Then the short-time Fourier transform has the following properties:
\begin{enumerate}
    \item The map $V_gf : \Rtd \to \C$ is uniformly continuous.
    \item The STFT satisfies the covariance property which reads
    $$
    V_g(T_\tau M_\nu f)(x,\omega) = e^{-2\pi i \tau \cdot \omega} V_gf(x-\tau,\omega-\nu).
    $$
    \item The $L^2(\Rtd)$-norm of $V_gf$ is given by $$\| V_gf\|_{L^2(\Rtd)} = \| f \|_{\ltd} \| g \|_{\ltd}.$$ In particular, $V_g : \ltd \to L^2(\Rtd)$ is a bounded linear operator.
\end{enumerate}
Moreover, the operators $\ft, T_\tau,M_\nu$ and $\mathcal{R}$ obey the relations
\begin{enumerate}
    \item[(4)] $\overline{\ft f} = \mathcal{R}\ft \overline{f}$.
    \item[(5)] $\ft \mathcal{R} = \mathcal{R} \ft$.
    \item[(6)] $\ft T_\tau = M_{-\tau} \ft$.
    \item[(7)] $\mathcal{R} T_\tau = T_{-\tau} \mathcal{R}$.
\end{enumerate}
\end{lemma}

\subsection{Symplectic matrices and metaplectic operators}\label{sec:symplectic_geometry}

For $d \in \N$ we denote by $I_d$ the identity matrix in $\R^{d \times d}$. The matrix $\mathcal{J} \in \R^{2d\times 2d}$, defined by
$$
\mathcal{J} \coloneqq \begin{pmatrix} 0 & -I_d \\ I_d & 0 \end{pmatrix},
$$
is called the \emph{standard symplectic matrix}. Using the standard symplectic matrix, we can define the symplectic group.

\begin{definition}
A matrix $S \in \mathrm{GL}_{2d}(\R)$ is called symplectic if 
$$
S^T \mathcal{J} S = \mathcal{J}.
$$
The set of all symplectic matrices is called the (real) symplectic group and is denoted by $\mathrm{Sp}_{2d}(\R)$.
\end{definition}

The property of a matrix $S \in \mathrm{GL}_{2d}(\R)$ being symplectic can be characterized in the following way \cite[Lemma 9.4.2]{Groechenig}.

\begin{lemma}\label{lma:symplecticCharacterization}
Let $S \in \mathrm{GL}_{2d}(\R)$ be an invertible matrix such that
$$
S = \begin{pmatrix} A & B \\ C & D \end{pmatrix}
$$
with $A,B,C,D$ denoting block matrices in $\R^{d \times d}$. Then $S$ is symplectic if and only if $AC^T = A^TC$, $BD^T=B^TD$ and $A^TD - C^TB = I_d$.
\end{lemma}

In the case $d=1$, Lemma \ref{lma:symplecticCharacterization} implies that a symplectic matrix $S \in \mathrm{Sp}_2(\R)$ with
$$
S= \begin{pmatrix} a & b \\ c & d\end{pmatrix}, \ \ a,b,c,d \in \R,
$$
satisfies the relation $ad-cb=1$. Hence, $\mathrm{Sp}_2(\R)$ coincides with the special linear group, $\mathrm{Sp}_2(\R) = \mathrm{SL}_{2}(\R)$.
The metaplectic group $\mathrm{Mp}(d)$ is the unitary representation of the double cover of the symplectic group $\mathrm{Sp}_{2d}(\R)$ on the Hilbert space $\ltd$. One can define it by requiring that the sequence of group homomorphisms
$$
0 \to \Z_2 \to \mathrm{Mp}(d) \to \mathrm{Sp}_{2d}(\R) \to 0
$$
be exact. Hence, to every symplectic matrix $S \in \mathrm{Sp}_{2d}(\R)$ there corresponds a pair of unitary operators $U_1,U_2 : \ltd \to \ltd$ differing by a sign, $U_1=-U_2$. Following the notation introduced in \cite{Folland+2016,Groechenig}, we denote this operator (defined up to a sign) by $\mu(S)$ and call it the metaplectic operator corresponding to $S$. Further, if $S,M \in \mathrm{Sp}_{2d}(\R)$ are two symplectic matrices, then, according to Schur's lemma, the metaplectic operator corresponding to the product of $S$ and $M$ satisfies
$$
\mu(SM) = c \mu(S)\mu(M), \ \ c \in \{ 1,-1\}.
$$
Hence, $\mu$ can be regarded as a group homomorphism up to a global phase. Further, we have the relation
$$
\mu(S)^{-1} = c \mu(S^{-1}), \ \ c \in \{ 1,-1\}.
$$
For certain symplectic block-diagonal matrices one can write out the corresponding metaplectic operator in a concise form \cite[Example 9.4.1(b)]{Groechenig}.

\begin{lemma}\label{lma:block_diag_representation}
For $A \in \mathrm{GL}_d(\R)$ define the block-diagonal matrix $S_A$ via
$$
S_A = \begin{pmatrix}
A & 0 \\
0 & A^{-T}
\end{pmatrix}
$$
where $A^{-T} = (A^T)^{-1}$. Then the metaplectic operator $\mu(S_A)$ corresponding to $S_A$ is given by
$$
(\mu(S_A)f)(x) = |\det A |^{-1/2} f(A^{-1}x), \ \ f \in \ltd.
$$
\end{lemma}

In addition to the previous lemma, the present exposition makes extensive use of the following interaction property of symplectic matrices, metaplectic operators, and the STFT \cite[Lemma 9.4.3]{Groechenig}.

\begin{lemma}\label{lma:interaction_property}
Let $f,g \in \ltd$ and let $S \in \mathrm{Sp}_{2d}(\R)$ be a symplectic matrix. Then for every $(x,\omega) \in \R^{2d}$ we have
$$
V_gf(S(x,\omega)) = e^{\pi i (x \cdot \omega - x' \cdot \omega')} V_{\mu(S)^{-1}g}(\mu(S)^{-1}f)(x,\omega)
$$
where $(x',\omega') = S(x,\omega)$ is the image of $(x,\omega)$ under $S$.
\end{lemma}

For a detailed study of symplectic matrices and metaplectic operators as well as the significance of symplectic geometry in quantum mechanics, the reader may consult the books by de Gosson \cite{deGosson} and Folland \cite{Folland+2016}.

\subsection{Lattices}

A set $\Lambda \subseteq \Rd$ is called a lattice if there exists an invertible matrix $A \in \mathrm{GL}_d(\R)$ such that $\Lambda = A\Z^d = \{ Az : z \in \Z^d \}$. The matrix $A$ is called the generating matrix of $\Lambda$ and $\Lambda$ is called generated by $A$. The reciprocal lattice (or: dual lattice) of $\Lambda$, denoted by $\Lambda^*$, is the lattice generated by $A^{-T} \coloneqq (A^T)^{-1}$. Equivalently, it is the set of all points $\lambda^* \in \Rd$ which satisfy
$$
\lambda \cdot \lambda^* \in \Z \ \ \ \forall \lambda \in \Lambda.
$$
The density of a lattice $\Lambda = A\Z^d$ is given by $|\det A|^{-1}$, a quantity independent of the generating matrix $A$.

Special lattices deserve special names. If $\mathcal{A},\mathcal{B}\subseteq \Rd$ are lattices and if $\Lambda = \mathcal{A} \times \mathcal{B}$ then, $\Lambda$ is called a separable lattice in $\Rtd$. If $\Lambda$ is generated by a diagonal matrix then we call it rectangular. Finally, $\Lambda \subseteq \R^{2d}$ is called symplectic if $\Lambda = \alpha S \Z^{2d}$ for some $\alpha \in \R \setminus \{ 0 \}$ and some $S \in \mathrm{Sp}_{2d}(\R)$. Thus, a symplectic lattice is generated by a scaled symplectic matrix.

Every lattice $\Lambda \subseteq \Rd$ gives rise to a tiling of $\Rd$ in terms of its fundamental domain. Denoting the standard cube in $\Rd$ by $Q_d \coloneqq [0,1]^d$ and assuming that $\Lambda$ is generated by $A \in \mathrm{GL}_d(\R)$,
$$
\mathcal{P}(\Lambda) \coloneqq AQ_d = \left \{ A(x_1, \dots, x_d) : x_j \in [0,1] \ \forall j \in \{ 1, \dots, d \} \right \}
$$
is called a fundamental domain of $\Lambda$. Further, we have
$$
\Rd = \bigcup_{\lambda \in \Lambda} \lambda + \mathcal{P}(\Lambda),
$$
and the intersection of two shifted fundamental domains $\lambda + \mathcal{P}(\Lambda)$ and $\lambda' + \mathcal{P}(\Lambda)$ with $\lambda,\lambda' \in \Lambda$ has $d$-dimensional Lebesgue measure zero. Note that $\mathcal{P}(\Lambda)$ depends on the choice of the generating matrix $A$. However, in the present exposition the exact shape of the fundamental domain of a lattice will not play a role, and all statements are valid for an arbitrary choice of a fundamental domain.

\subsection{Shift-invariant spaces associated to metaplectic operators}

Let $\Lambda \subseteq \Rd$ be a lattice, and let $\phi \in \ltd$. The (principle) shift-invariant space generated by $\phi$ and subjected to the lattice $\Lambda$ is defined as
\begin{equation}\label{classical_si_space}
    \mathcal{V}_\Lambda(\phi) \coloneqq \overline{\lspan \{ T_\lambda \phi : \lambda \in \Lambda \}}
\end{equation}
where $T_\lambda$ denotes the shift-operator, $T_\lambda \phi = \phi(\cdot - \lambda)$. The closure is taken with respect to the $L^2$-norm. Shift-invariance means that $T_\lambda h \in \mathcal{V}_\Lambda(\phi)$ provided that $h \in \mathcal{V}_\Lambda(\phi)$ and $\lambda \in \Lambda$.
A generalization to the setting where the space is invariant under the operator $\mu(S)T_\lambda\mu(S)^{-1}$ for some symplectic matrix $S \in \mathrm{Sp}_{2d}(\R)$ is the content of the following definition.

\begin{definition}
Let $\Lambda \subseteq \Rd$ be a lattice, let $S \in \mathrm{Sp}_{2d}(\R)$ be a symplectic matrix, and let $\phi \in \ltd$ be a generating function. The $S$-shift of $\phi$ by $\lambda \in \Lambda$ is defined by
$$
T_\lambda^S \phi \coloneqq \mu(S)T_\lambda\mu(S)^{-1} \phi.
$$
Further, we call the space
$$
\mathcal{V}_\Lambda^S(\phi) \coloneqq \overline{\lspan\{ T_\lambda^S \phi : \lambda \in \Lambda \}}
$$
the shift-invariant space generated by $\phi$, associated to the $S$-shift, and subjected to the lattice $\Lambda$ (for short, $(S,\Lambda)$-shift invariant space).
\end{definition}

In equation \eqref{classical_si_space} we defined the $(S,\Lambda)$-shift-invariant space $\mathcal{V}_\Lambda^S(\phi)$ as the $L^2$-closure of the $S$-shifts by $\lambda \in \Lambda$ of the generating function $\phi$. Frequently, shift-invariant subspaces of $\ltd$ are defined as the set of all functions
\begin{equation}\label{series}
    f = \sum_{\lambda \in \Lambda} c_\lambda T_\lambda^S \phi
\end{equation}
such that $\{ c_\lambda \}$ is a square-summable sequence, $\{ c_\lambda \} \in \ell^2(\Lambda)$. This definition is used if the generator $\phi$ satisfies certain regularity properties which guarantee convergence of the series in \eqref{series}. In the present exposition we aim for results of the highest generality. Thus, in our case no prior assumptions on $\phi$ are made. However, whenever $f \in \mathcal{V}_\Lambda^S(\phi)$ can be written as a convergent series of the form $f = \sum_{\lambda \in \Lambda} c_\lambda T_\lambda^S \phi$ with $\{ c_\lambda \} \in \ell^2(\Lambda)$ then we call $\{ c_\lambda \}$ a \emph{defining sequence} of $f$. Note that if $\{ c_\lambda \} \in c_{00}(\Lambda)$, the space of all sequences with only finitely many non-zero components, then $f = \sum_{\lambda \in \Lambda} c_\lambda T_\lambda^S \phi$ is a well-defined element in $\mathcal{V}_\Lambda^S(\phi)$ and no questions of convergence appear. If the system $\{ T_\lambda^S \phi : \lambda \in \Lambda \}$ forms a Bessel sequence in $\ltd$ then standard theory implies that the series $\sum_{\lambda \in \Lambda} c_\lambda T_\lambda^S \phi$ converges unconditionally for every $\{ c_\lambda \} \in \ell^2(\Lambda)$ (see, for instance, \cite[Corollary 3.2.5]{christensenBook}). Recall that a sequence $\{ f_n : n \in \N \} \subseteq X$ of a separable Hilbert space $X$ with inner product $\langle \cdot , \cdot \rangle_X$ and induced norm $\| \cdot \|_X$ is called a Bessel sequence if there exists a constant $B>0$ such that
$$
\sum_{n=1}^\infty |\langle f,f_n \rangle_X|^2 \leq B \| f \|_X^2 \ \ \ \forall f \in X. 
$$
The property of $\{ T_\lambda^S \phi : \lambda \in \Lambda \}$ being a Bessel sequence in $\ltd$ can be phrased in terms of boundedness of a certain lattice periodization. To that end, for a map $\phi \in \ltd$ we introduce the $\Lambda$-periodization of $|\ft \phi|^2$ by
$$
\mathfrak{p}_\Lambda[\phi](t) \coloneqq  \sum_{\lambda \in \Lambda} |\ft \phi (t+\lambda)|^2.
$$

\begin{proposition}\label{prop:bessel_condition}
Let $\phi \in \ltd$ and let $\Lambda \subseteq \Rd$ be a lattice. If $\mathfrak{p}_{\Lambda^*}[\phi] \in L^\infty(\mathcal{P}(\Lambda^*))$ then $\{ T_\lambda \phi : \lambda \in \Lambda \}$ is a Bessel sequence, and the series $$\sum_{\lambda \in \Lambda} c_\lambda T_\lambda \phi$$ converges unconditionally for every $\{ c_\lambda \} \in \ell^2(\Lambda)$.
\end{proposition}
\begin{proof}
The proof is similar to the one presented in \cite[Theorem 9.2.5]{christensenBook}. For the convenience of the reader we include a proof of the statement in the Appendix \ref{appendix:A}.
\end{proof}

A consequence of the previous proposition is the following result.

\begin{corollary}\label{cor:besselCorollary}
Let $\phi \in \ltd$, let $\Lambda \subseteq \Rd$ be a lattice, and let $S \in \mathrm{Sp}_{2d}(\R)$ be a symplectic matrix. If $\mathfrak{p}_{\Lambda^*}[\mu(S)^{-1}\phi] \in L^\infty(\mathcal{P}(\Lambda^*))$ then $\{ T_\lambda^S \phi : \lambda \in \Lambda \}$ is a Bessel sequence, and the series $\sum_{\lambda \in \Lambda} c_\lambda T_\lambda^S \phi$ converges unconditionally for every $\{ c_\lambda \} \in \ell^2(\Lambda)$. 
\end{corollary}
\begin{proof}
Proposition \ref{prop:bessel_condition} implies that under the conditions stated above, the system $\{ T_\lambda \mu(S)^{-1} \phi : \lambda \in \Lambda \}$ is a Bessel sequence in $\ltd$. Since the property of being a Bessel sequence is invariant under the action of a bounded linear operator, it follows that
$$
\{ \mu(S)T_\lambda\mu(S)^{-1} \phi : \lambda \in \Lambda \} = \{ T_\lambda^S \phi : \lambda \in \Lambda \}
$$
is a Bessel sequence. In particular, $\sum_{\lambda \in \Lambda} c_\lambda T_\lambda^S \phi$ converges unconditionally for every $\{ c_\lambda \} \in \ell^2(\Lambda)$.
\end{proof}

\section{Main results}\label{sec:main_results_3}

The objective of this section is to state and prove the main results of the article. Section \ref{sec:spec} builds upon the definition of shift-invariant spaces associated to metaplectic operators. It will be shown, that functions belonging to such spaces are candidates for counterexamples for the sampled STFT phase retrieval problem, that is, they give rise to a function $f$ and a function which will be denoted by $f_\times$, such that $f$ and $f_\times$ produce identical spectrogram samples on certain determined lattices in the time-frequency plane. In Section \ref{sec:noneq}, we investigate under which assumptions the functions $f$ and $f_\times$ are equivalent up to a global phase and prove a corresponding characterization. Sections \ref{sec:implications} -- \ref{univariate_case} are devoted to the proof of the main results of the article as stated in Section \ref{sec:contribution}. In addition, we provide a series of visualizations.

\subsection{Spectrogram equalities}\label{sec:spec}

Let $\Lambda \subseteq \Rd$ be a lattice, $S \in \mathrm{Sp}_{2d}(\R)$ be a symplectic matrix, $\phi \in \ltd$ be a generating function, and $f \in \mathcal{V}_\Lambda^S(\phi)$. Whenever $f$ has defining sequence $\{ c_\lambda \}$, i.e., $f$ may be written as a convergent series of the form $f = \sum_{\lambda \in \Lambda} c_\lambda T_\lambda^S \phi$, we define a corresponding function $f_\times$ via
$$
f_\times \coloneqq \sum_{\lambda \in \Lambda} \overline{c_\lambda} T_\lambda^S \phi.
$$
The map $f_\times$ arises from $f$ via complex conjugation of its defining sequence $\{ c_\lambda \}$. We start by inspecting the spectrograms of $f$ and $f_\times$ on sets of the form $S(A \times B) = \{ S(a,b) : a \in A, b \in B \} \subseteq \R^{2d}$ with $A,B \subseteq \Rd$ and $S \in \mathrm{Sp}_{2d}(\R)$.

\begin{theorem}\label{thm:equalSpectrograms}
Let $g \in \ltd$ be a window function, $\Lambda \subseteq \Rd$ be a lattice, and $S \in \mathrm{Sp}_{2d}(\R)$ be a symplectic matrix. Suppose that $\{ c_\lambda \} \subseteq \C$ is the defining sequence of $f \in \mathcal{V}_\Lambda^S(\mathcal{R}g)$. Then the following holds:
\begin{enumerate}
    \item If $\{ c_\lambda \} \in c_{00}(\Lambda)$ then
$$
|V_gf(z)| = |V_g(f_\times)(z)| \ \ \forall z \in S(\Rd \times \Lambda^*).
$$
\item  If $\mathfrak{p}_{\Lambda^*}[\mu(S)^{-1}\mathcal{R}g] \in L^\infty(\mathcal{P}(\Lambda^*))$ and $\{ c_\lambda \} \in \ell^2(\Lambda)$ then
$$
|V_gf(z)| = |V_g(f_\times)(z)| \ \ \forall z \in S(\Rd \times \Lambda^*).
$$
\end{enumerate}
\end{theorem}
\begin{proof}
\textbf{Proof of (1).}
Suppose that $f = \sum_{\lambda \in \Lambda} c_\lambda T_\lambda^S \mathcal{R}g \in \mathcal{V}_\Lambda^S(\mathcal{R}g)$ has defining sequence $\{ c_\lambda \} \in c_{00}(\Lambda)$. For $(x,\omega) \in \Rtd$ we obtain the relations
\begin{align}\label{eq:specF}
        |V_gf(S(x,\omega))|  & \textoverset[0]{Lemma \ref{lma:interaction_property}}{=} \left | V_{\mu(S)^{-1} g} \left ( \sum_{\lambda \in \Lambda} c_\lambda T_{\lambda} \mu(S)^{-1} \mathcal{R}g \right ) (x,\omega) \right | \nonumber \\
          & \textoverset[0]{linearity}{=} \left |  \sum_{\lambda \in \Lambda} c_\lambda V_{\mu(S)^{-1} g} (T_\lambda \mu(S)^{-1} \mathcal{R}g) (x,\omega) \right | \nonumber \\
          & \textoverset[0]{Lemma \ref{lma:stft_properties}(2)}{=} \left |  \sum_{\lambda \in \Lambda} c_\lambda e^{-2 \pi i \lambda \cdot \omega} V_{\mu(S)^{-1} g}(\mu(S)^{-1}\mathcal{R}g)(x-\lambda,\omega)    \right | \nonumber \\
          & \textoverset[0]{}{=} \left | \sum_{\lambda \in \Lambda} c_\lambda e^{-2 \pi i \lambda \cdot \omega} \ft ( (\mu(S)^{-1}\mathcal{R}g) ( \overline{T_{x-\lambda} \mu(S)^{-1} g } )  )(\omega) \right |. 
\end{align}
If we replace $c_\lambda$ with $\overline{c_\lambda}$ in equation \eqref{eq:specF}, we obtain the identity
\begin{align}\label{eq:specF_cross1}
        |V_g(f_\times)(S(x,\omega))|  & \textoverset[0]{}{=} \left | \sum_{\lambda \in \Lambda} \overline{c_\lambda} e^{-2 \pi i \lambda \cdot \omega} \ft ( (\mu(S)^{-1}\mathcal{R}g) ( \overline{T_{x-\lambda} \mu(S)^{-1} g } )  )(\omega) \right | \nonumber  \\
        & \textoverset[0]{Lemma \ref{lma:stft_properties}(4)}{=}  \left | \sum_{\lambda \in \Lambda} c_\lambda e^{2 \pi i \lambda \cdot \omega} \mathcal{R} \ft ( (\overline{\mu(S)^{-1}\mathcal{R}g}) (T_{x-\lambda} \mu(S)^{-1} g )  )(\omega) \right |. 
\end{align}
We now inspect the operator $\mu(S)^{-1}\mathcal{R}$ appearing in equation \eqref{eq:specF_cross1}. Note that Lemma \ref{lma:block_diag_representation} shows that for every $h \in \ltd$ we have
$$
(\mu(-I_{2d}) h)(x) = |\det(-I_{2d})|^{-1/2}h(-x) = \mathcal{R}h(x).
$$
Consequently, the metaplectic operator corresponding to the negative of the identity matrix in $\R^{2d \times 2d}$ is the reflection operator $\mathcal{R}$. Using the properties of the map $\mu$ as given in Section \ref{sec:symplectic_geometry} and the fact that every $S \in \mathrm{Sp}_{2d}(\R)$ commutes with $-I_{2d}$, it follows that there exists a constant $\nu \in \{ 1,-1\}$ such that
\begin{equation}\label{R_commuting}
    \mu(S)^{-1}\mathcal{R} = \nu \mathcal{R}\mu(S)^{-1}.
\end{equation}
In other words, $\mu(S)^{-1}$ and $\mathcal{R}$ commute up to a global phase. Proceeding with equation \eqref{eq:specF_cross1} yields
\begin{align}\label{eq:specF_cross2}
        |V_g(f_\times)(S(x,\omega))|  & \textoverset[0]{equation \eqref{R_commuting}}{=}  \left | \sum_{\lambda \in \Lambda} c_\lambda e^{2 \pi i \lambda \cdot \omega} \mathcal{R} \ft ( ( \mathcal{R} \overline{ \mu(S)^{-1} g} ) (T_{x-\lambda} \mu(S)^{-1} g )  )(\omega) \right | \nonumber  \\
        & \textoverset[0]{Lemma \ref{lma:stft_properties}(5)}{=}  \left | \sum_{\lambda \in \Lambda} c_\lambda e^{2 \pi i \lambda \cdot \omega} \ft ( ( \overline{\mu(S)^{-1} g} ) (\mathcal{R} T_{x-\lambda} \mu(S)^{-1} g )  )(\omega) \right | \nonumber  \\
        & \textoverset[0]{Lemma \ref{lma:stft_properties}(7)}{=}  \left | \sum_{\lambda \in \Lambda} c_\lambda e^{2 \pi i \lambda \cdot \omega} \ft ( (\overline{\mu(S)^{-1} g} ) (T_{\lambda-x} \mathcal{R} \mu(S)^{-1} g )  )(\omega) \right | \nonumber \\
        & \textoverset[0]{equation \eqref{R_commuting}}{=}  \left | \sum_{\lambda \in \Lambda} c_\lambda e^{2 \pi i \lambda \cdot \omega} \ft ( (\overline{\mu(S)^{-1} g} ) (T_{\lambda-x} \mu(S)^{-1} \mathcal{R} g )  )(\omega) \right | \nonumber  \\
        & \textoverset[0]{}{=}  \left | \sum_{\lambda \in \Lambda} c_\lambda e^{2 \pi i \lambda \cdot \omega} \ft ( T_{\lambda-x} [ ( T_{x-\lambda} \overline{\mu(S)^{-1} g}) ( \mu(S)^{-1} \mathcal{R} g ) ]  )(\omega) \right | \nonumber \\
        & \textoverset[0]{Lemma \ref{lma:stft_properties}(6)}{=}  \left | \sum_{\lambda \in \Lambda} c_\lambda e^{2 \pi i \lambda \cdot \omega} e^{-2\pi i \omega \cdot (\lambda-x)} \ft (  ( T_{x-\lambda} \overline{\mu(S)^{-1} g}) ( \mu(S)^{-1} \mathcal{R} g ))(\omega) \right | \nonumber \\
         & \textoverset[0]{}{=}  \left | \sum_{\lambda \in \Lambda} c_\lambda \ft (  ( \overline{ T_{x-\lambda} \mu(S)^{-1} g}) ( \mu(S)^{-1} \mathcal{R} g ))(\omega) \right |.
\end{align}
Now suppose that $\omega = \lambda^* \in \Lambda^*$ is an element of the reciprocal lattice of $\Lambda$. Then $ \lambda \cdot \lambda^*  \in \Z$ and therefore the phase factor $e^{-2 \pi i \lambda \cdot \omega}$ appearing in equation \eqref{eq:specF} reduces to
$$
e^{-2 \pi i \lambda \cdot \omega} = e^{-2 \pi i \lambda \cdot \lambda^*} = 1.
$$
Comparing equation \eqref{eq:specF} with equation \eqref{eq:specF_cross2} shows that in the case $\omega \in \Lambda^*$ we have $|V_gf(S(x,\omega))| = |V_g(f_\times)(S(x,\omega))|$. Since $x \in \Rd$ was arbitrary, we conclude that the spectrogram of $f$ agrees with the spectrogram of $f_\times$ on the set $S(\Rd \times \Lambda^*)$.

\textbf{Proof of (2).}
Now assume that  $\mathfrak{p}_{\Lambda^*}[\mu(S)^{-1}\mathcal{R}g] \in L^\infty(\mathcal{P}(\Lambda^*))$. According to Corollary \ref{cor:besselCorollary} the series $f = \sum_{\lambda \in \Lambda} c_\lambda T_\lambda^S \mathcal{R}g$ converges unconditionally to an element in $\mathcal{V}_\Lambda^S(\mathcal{R}g)$ whenever $\{ c_\lambda \} \in \ell^2(\Lambda)$. Since the STFT is a continuous operator on $\ltd$, this implies that the chain of equalities leading to equation \eqref{eq:specF} is justified under the weaker assumption that $\{ c_\lambda \} \in \ell^2(\Lambda)$ since one can interchange summation with the application of the STFT. From this point, the proof of statement (2) follows analogously from the proof of statement (1).
\end{proof}

The foregoing theorem dealt with equality of two spectrograms on sets of the form $S(\Rd \times \mathcal{B})$ with $\mathcal{B}$ as lattice. Replacing $S(\Rd \times \mathcal{B})$ by $S(\mathcal{A} \times \mathcal{B})$ with $\mathcal{A} \subseteq \Rd$ as lattice, leads to sufficient conditions on matrices $L\in \mathrm{GL}_{2d}(\R)$ so that the spectrogram of $f$ and $f_\times$ agree on the lattice generated by $L$.

\begin{corollary}\label{cor:ST}
Let $L \in \mathrm{GL}_{2d}(\R)$ be an invertible matrix which factors into $L=ST$ where $S \in \mathrm{Sp}_{2d}(\R)$ is a symplectic matrix and $T \in \mathrm{GL}_{2d}(\R)$ is a block-diagonal matrix of the form
$$
T = \begin{pmatrix} A & 0 \\ 0 & B \end{pmatrix}, \ \ A,B \in \mathrm{GL}_d(\R).
$$
Further, let $\mathcal{L}=L\Z^{2d}$ be the lattice generated by $L$ and let $\mathcal{B}=B\Z^d$ be the lattice generated by $B$.
Let $g \in \ltd$ be a window function, and suppose that $\{ c_\lambda : \lambda \in \mathcal{B}^* \} \subseteq \C$ is the defining sequence of $f \in \mathcal{V}_{\mathcal{B}^*}^S(\mathcal{R}g)$. Then the following holds:
\begin{enumerate}
    \item If $\{ c_\lambda \} \in c_{00}(\mathcal{B}^*)$ then
$$
|V_gf(z)| = |V_g(f_\times)(z)| \ \ \forall z \in \mathcal{L}.
$$
\item If $\mathfrak{p}_\mathcal{B}(\mu(S)^{-1}\mathcal{R}g) \in L^\infty(\mathcal{P}(\mathcal{B}))$ and $\{ c_\lambda \} \in \ell^2(\mathcal{B}^*)$ then 
$$
|V_gf(z)| = |V_g(f_\times)(z)| \ \ \forall z \in \mathcal{L}.
$$
\end{enumerate}
\end{corollary}
\begin{proof}
\textbf{Proof of (1).}
Suppose that $\{ c_\lambda \} \in c_{00}(\mathcal{B}^*)$ is the defining sequence of $f \in \mathcal{V}_{\mathcal{B}^*}^S(\mathcal{R}g)$. Since $(\mathcal{B}^*)^* = \mathcal{B}$, Theorem \ref{thm:equalSpectrograms}(1) and the choice of $f$ implies that
\begin{equation}\label{eq:s1}
    |V_gf(z)| = |V_g(f_\times)(z)| \ \ \forall z \in S(\Rd \times \mathcal{B}).
\end{equation}
The statement is therefore a consequence of the inclusion
\begin{equation}\label{eq:inclusion}
    \mathcal{L}=ST\Z^{2d} = S(A\Z^d \times B\Z^d) \subseteq S(\R^d \times B\Z^d) = S(\Rd \times \mathcal{B}).
\end{equation}
In view of identity \eqref{eq:s1}, the spectrogram of $f$ agrees with the spectrogram of $f_\times$ on the lattice $\mathcal{L}$.

\textbf{Proof of (2).}
To prove the second statement we observe that the condition $\mathfrak{p}_\mathcal{B}(\mu(S)^{-1}\mathcal{R}g) \in L^\infty(\mathcal{P}(\mathcal{B}))$ implies that $f = \sum_{\lambda \in \mathcal{B}^*} c_\lambda T_\lambda^S \mathcal{R}g$ converges unconditionally provided that $\{ c_\lambda \} \in \ell^2(\mathcal{B}^*)$. In a similar fashion to the proof of statement (1) we conclude from Theorem \ref{thm:equalSpectrograms}(2) that
$$
|V_gf(z)| = |V_g(f_\times)(z)| \ \ \forall z \in S(\Rd \times \mathcal{B}),
$$
thereby proving the statement by invoking the inclusion given in \eqref{eq:inclusion}.
\end{proof}

The univariate case $d=1$ deserves special attention. In this case,  Corollary \ref{cor:ST} covers all lattices in $\R^2$: starting from an arbitrary lattice $\mathcal{L} \subseteq \R^2$ we can construct functions $f$ and $f_\times$ so that their spectrograms agree on $\mathcal{L}$. This is a consequence of the fact that in $\R^2$ all lattices are symplectic, a property which holds exclusively in $\R^2$.

\begin{corollary}\label{cor:1d_lattice}
Let $\mathcal{L} \subseteq \R^2$ be an arbitrary lattice generated by $L \in \mathrm{GL}_2(\R)$ and let $g \in \lt$ be a window function. If $\alpha \coloneqq \det L$ then $S \coloneqq \alpha^{-1} L$ is a symplectic matrix. Moreover, if $\{ c_\lambda \} \subseteq \C$ is the defining sequence of $f \in \mathcal{V}_{\alpha^{-1}\Z}^S(\mathcal{R}g)$ then the following holds:
\begin{enumerate}
    \item If $\{c_\lambda \} \in c_{00}(\alpha^{-1}\Z)$ then
$$
|V_gf(z)| = |V_g(f_\times)(z)| \ \ \forall z \in \mathcal{L}.
$$
\item If $\mathfrak{p}_{\alpha \Z} (\mu(S)^{-1}\mathcal{R}g) \in L^\infty([0,\alpha])$ and if $\{ c_\lambda \} \in \ell^2(\alpha^{-1}\Z)$ then 
$$
|V_gf(z)| = |V_g(f_\times)(z)| \ \ \forall z \in \mathcal{L}.
$$
\end{enumerate}
\end{corollary}
\begin{proof}
\textbf{Proof of (1).}
Since $0 \neq \alpha = \det L$, the matrix $S=\alpha^{-1}L$ satisfies $\det S = 1$. Writing
$$
S= \begin{pmatrix} a & b \\ c & d\end{pmatrix}, \ \ a,b,c,d \in \R,
$$
we obtain the relation $ad-cd=1$. In view of Lemma \ref{lma:symplecticCharacterization}, this implies that $S \in \mathrm{GL}_2(\R)$ is a symplectic matrix. Now let $T=\alpha I_2 \in \R^{2 \times 2}$.
Then $L=ST$ and the reciprocal lattice of $\alpha\Z$ is given by $\alpha^{-1}\Z$. Thus, Corollary \ref{cor:ST} implies that if $f \in \mathcal{V}_{\alpha^{-1}\Z}^S(\mathcal{R}g)$ has defining sequence $\{ c_\lambda \} \in c_{00}(\alpha^{-1}\Z)$ then the spectrogram of $f$ and $f_\times$ agree on
$$
ST\Z^2 = L\Z^2 = \mathcal{L}.
$$

\textbf{Proof of (2).}
Suppose, on the other hand, that $\{ c_\lambda \} \in \ell^2(\alpha^{-1}\Z)$. Since $(\alpha^{-1}\Z)^* = \alpha \Z$ and sine $\mathcal{P}(\alpha\Z)=[0,\alpha] \subseteq \R$, the condition $\mathfrak{p}_{(\alpha^{-1}\Z)^*} (\mu(S)^{-1}\mathcal{R}g) \in L^\infty(\mathcal{P}((\alpha^{-1}\Z)^*))$ is equivalent to the condition $\mathfrak{p}_{\alpha \Z}(\mu(S)^{-1}\mathcal{R}g) \in L^\infty([0,\alpha])$. Under this assumption, Corollary \ref{cor:besselCorollary} implies that $f \in \mathcal{V}_{\alpha^{-1}\Z}^S(\mathcal{R}g)$ converges unconditionally provided that $f$ has defining sequence $\{ c_\lambda \} \in \ell^2(\alpha^{-1}\Z)$. The derivation of statement (2) then follows analogously to the derivation of statement (1).
\end{proof}

\subsection{Characterization of equality up to a global phase}\label{sec:noneq}

Let $\Lambda \subseteq \Rd$ be a lattice, $S \in \mathrm{Sp}_{2d}(\R)$ be a symplectic matrix, and $\phi \in \ltd$. In this section we seek to classify those defining sequences $\{ c_\lambda \} \subseteq \C$ of a function $f = \sum_{\lambda \in \Lambda} c_\lambda T_\lambda^S \phi \in \mathcal{V}_\Lambda^S(\phi)$, such that $f$ and its companion $f_\times$ do not agree up to a global phase, i.e., $f \nsim f_\times$. The obtained results serve as the machinery for the establishment of fundamental discretization barriers for the STFT phase retrieval problem. The statements derived in the present section are based on linear independence properties of systems of translates. These properties are given in Theorem \ref{thm:independence} below. Before stating and proving this result, we require an elementary uniqueness property of functions of several complex variables.

\begin{lemma}\label{lma:elementary_uniqueness}
Let $\mathcal{Y} \subseteq \Rd$ be a Lebesgue-measurable set of positive $d$-dimensional Lebesgue measure. If $F : \C^d \to \C$ is a holomorphic function of $d$ complex variables which vanishes on $\mathcal{Y}$ then $F$ vanishes identically.
\end{lemma}
\begin{proof}
See the Appendix \ref{appendix:B}.
\end{proof}

If $\Lambda \subseteq \Rd$ is a lattice, then we say that a function $f : \Rd \to \C$ is $\Lambda$-periodic if $f(t+\lambda)=f(t)$ for every $t \in \Rd$ and every $\lambda \in \Lambda$. Similarly, we say that a set $\Omega \subseteq \Rd$ is $\Lambda$-periodic if $\lambda + \Omega = \Omega$ for every $\lambda \in \Lambda$.

\begin{theorem}\label{thm:independence}
Let $0 \neq \phi \in \ltd$ and let $\Lambda \subseteq \Rd$ be a lattice. Then the following holds:
\begin{enumerate}
    \item The system of $\Rd$-shifts of $\phi$, i.e., the set
$
\{ T_x \phi : x \in \Rd \}
$
forms a (finitely) linearly independent system.
\item If there exist positive constants $0<A,B<\infty$ such that
$$
A \leq \mathfrak{p}_{\Lambda^*}[\phi](t) \leq B
$$
for almost every $t \in \mathcal{P}(\Lambda^*)$ then $\{ T_\lambda \phi : \lambda \in \Lambda \}$ is $\ell^2(\Lambda)$-independent, i.e., a sequence $\{c_\lambda \} \in \ell^2(\Lambda)$ is the zero sequence provided that $\sum_{\lambda \in \Lambda} c_\lambda T_\lambda \phi = 0$.
\end{enumerate}
\end{theorem}
\begin{proof}
\textbf{Proof of (1).}
We start by proving the claim made in item (1). To that end, let $\mathcal{X} \subset \Rd$ be a finite set, $|\mathcal{X}|<\infty$, and let $c_x \in \C$ with $x \in \mathcal{X}$. Suppose that
\begin{equation}\label{eq:lin_indep_zero}
    \sum_{x \in \mathcal{X}} c_x T_x \phi = 0.
\end{equation}
We have to show that $c_x = 0$ for every $x \in \mathcal{X}$. To do so, we take the Fourier transform on both sides of equation \eqref{eq:lin_indep_zero} and obtain the relation
$$
 \underbrace{\sum_{x \in \mathcal{X}} c_x e^{-2\pi i x \cdot t}}_{\coloneqq E(t)} \ft \phi(t) = 0.
$$
which holds for every $t \in \R$. Since $\phi\neq 0$ there exists a set $\mathcal{Y} \subseteq \Rd$ of positive $d$-dimensional Lebesgue measure such that $\ft \phi (t) \neq 0$ for almost every $t \in \mathcal{Y}$. Hence, the map $E$ must vanish almost everywhere on $\mathcal{Y}$. But $E$ extends from $\Rd$ to a holomorphic function of $d$ complex variables. Therefore, Lemma \ref{lma:elementary_uniqueness} shows that $E$ must vanish identically. Now observe that since $\mathcal{X}$ is a set, all its elements are distinct. A general theorem on the linear independence of characters due to Artin \cite[Theorem 4.1]{lang2005algebra} shows that the complex exponentials $t \mapsto e^{-2\pi i t \cdot x}, x \in \mathcal{X}$, are linearly independent. Hence, the property that $E$ vanishes identically implies that $c_x=0$ for all $x \in \mathcal{X}$ and this concludes the proof of the claim.

\textbf{Proof of (2).}
We continue with the proof of the second statement. First, observe that since $\mathfrak{p}_{\Lambda^*}[\phi] \in L^\infty(\mathcal{P}(\Lambda^*))$ it follows from Proposition \ref{prop:bessel_condition} that the series $\sum_{\lambda \in \Lambda} c_\lambda T_\lambda \phi$ converges unconditionally for every $\{ c_\lambda \} \in \ell^2(\Lambda)$. Now suppose that $\{ c_\lambda \} \in \ell^2(\Lambda)$ is such that $\sum_{\lambda \in \Lambda} c_\lambda T_\lambda \phi = 0$. By taking the Fourier transform, we obtain the relation
\begin{equation}\label{G_zero}
    \underbrace{\sum_{\lambda \in \Lambda} c_\lambda e^{-2 \pi i \lambda \cdot t}}_{\coloneqq G(t)}  \ft \phi(t) = 0
\end{equation}
which holds for almost every $t \in \Rd$. By assumption we have that
\begin{equation}\label{eq:lower_triangle_bound}
    \mathfrak{p}_{\Lambda^*}[\phi](t) = \sum_{\lambda^* \in \Lambda^*} |\ft \phi (t+\lambda^*)|^2 \geq A >0
\end{equation}
for almost every $t \in \mathcal{P}(\Lambda^*)$. Since $\mathfrak{p}_{\Lambda^*}[\phi]$ is $\Lambda^*$-periodic, it follows that the lower bound given in equation \eqref{eq:lower_triangle_bound} holds for almost every $t \in \Rd$. Moreover, it shows that there exists no measurable $\Lambda^*$-periodic set $\Omega \subseteq \Rd$ of positive Lebesgue measure such that $\ft \phi$ vanishes almost everywhere on $\Omega$. Since the map $G$ is $\Lambda^*$-periodic as well, equation \eqref{G_zero} implies that $G$ must vanish almost everywhere. The uniqueness theorem for Fourier coefficients finally shows that $c_\lambda = 0$ for every $\lambda \in \Lambda$, as desired.
\end{proof}

In order to characterize equivalence up to a global phase we require the following observation, which was proved in \cite[Lemma 3.4]{grohsLiehrJFAA}.

\begin{lemma}\label{lemma:funcion_values_lines}
Let $Y$ be an arbitrary set, and let $f : Y \to \C$ be a complex-valued map. Then $f \sim \overline{f}$ if and only if there exists an $\alpha \in \R$ such that $f(y) \in e^{i\alpha}\R$ for every $y \in Y$.
\end{lemma}

The previous lemma characterizes equivalence of $f$ and $\overline{f}$ in terms of function values on a line $e^{i\alpha}\R \subseteq \C, \alpha \in \R$. If $\Lambda \subseteq \Rd$ is a lattice, then every sequence $\{ c_\lambda \} \in \ell^2(\Lambda)$ is a function from $\Lambda$ to $\C$ and we can consider those sequences $\{ c_\lambda \} \in \ell^2(\Lambda)$ whose values may or may not lie on a line in the complex plane passing through the origin. This motivates the following definition.

\begin{definition}
Let $\Lambda \subseteq \Rd$ be a lattice. We define the class of square-summable sequences with index set $\Lambda$ whose elements do not lie on a line in the complex plane passing through the origin by
$$
\ell^2_{\mathcal{O}}(\Lambda) \coloneqq \left \{ \{ c_\lambda \} \in \ell^2(\Lambda) : \nexists \, \alpha \in \R \ \mathrm{s.t.} \ \{ c_\lambda \} \subseteq e^{i\alpha} \R \right \}.
$$
\end{definition}

Note that if $\{ c_\lambda \} \in \ell^2(\Lambda) \setminus \ell^2_{\mathcal{O}}(\Lambda)$ then $\{ c_\lambda \}$ takes values on a line in the complex plane passing through the origin, i.e., there exists an $\alpha \in \R$ such that $c_\lambda \in e^{i\alpha}\R$ for every $\lambda \in \Lambda$. In other words, $\{ c_\lambda \}$ is a real sequence up to a global phase.

\begin{theorem}\label{thm:nonEquivalence}
Let $0 \neq \phi \in \ltd$, let $\Lambda \subseteq \Rd$ be a lattice, and let $S \in \mathrm{Sp}_{2d}(\R)$ be a symplectic matrix. Suppose that $\{c_\lambda \} \subseteq \C$ is the defining sequence of $f \in \mathcal{V}_\Lambda^S(\phi)$. Then the following holds:
\begin{enumerate}
    \item If $\{c_\lambda \} \in c_{00}(\Lambda)$ then 
    $$
    f \nsim f_\times \iff \{ c_\lambda \} \in \ell^2_\mathcal{O}(\Lambda).
    $$
    \item If there exist positive constants $0<A,B<\infty$ such that
$$
A \leq \mathfrak{p}_{\Lambda^*}[\mu(S)^{-1}\phi](t) \leq B
$$
for almost every $t \in \mathcal{P}(\Lambda^*)$ then
$$
    f \nsim f_\times \iff \{ c_\lambda \} \in \ell^2_\mathcal{O}(\Lambda).
    $$
\end{enumerate}
\end{theorem}
\begin{proof}
\textbf{Proof of (1).} We start by proving that $\{c_\lambda \} \in c_{00}(\Lambda) \cap \ell^2_\mathcal{O}(\Lambda)$ implies that $f \nsim f_\times$. To this end, let $\{ c_\lambda \} \in \ell^2_\mathcal{O}(\Lambda) \cap c_{00}(\Lambda)$ and suppose by contradiction that $f \sim f_\times$. Then there exists a constant $\nu \in \T$ such that
$$
\sum_{\lambda \in \Lambda} c_\lambda \mu(S)T_\lambda\mu(S)^{-1} \phi = \nu \sum_{\lambda \in \Lambda} \overline{c_\lambda} \mu(S)T_\lambda\mu(S)^{-1} \phi.
$$
Since $\mu(S)$ is invertible, it follows that
$$
\sum_{\lambda \in \Lambda} (c_\lambda - \nu \overline{c_\lambda} )T_\lambda\mu(S)^{-1} \phi = 0.
$$
By assumption we have $\phi \neq 0$ and therefore the invertibility of $\mu(S)$ implies that $\mu(S)^{-1} \phi \neq 0$. In view of Theorem \ref{thm:independence} the system of translates $\{ T_x \mu(S)^{-1} \phi : x \in \Rd \}$ forms a finitely linearly independent system. Since $\{ c_\lambda \} \in c_{00}(\Lambda)$ there exists a finite set $U \subseteq \Lambda$ such that
$$
\sum_{\lambda \in \Lambda} (c_\lambda - \nu \overline{c_\lambda} )T_\lambda\mu(S)^{-1} \phi = \sum_{\lambda \in U} (c_\lambda - \nu \overline{c_\lambda} )T_\lambda\mu(S)^{-1} \phi = 0.
$$
Linear independence yields $c_\lambda - \nu \overline{c_\lambda} = 0$ for all $\lambda \in U$. Hence, the map $U \to \C, \lambda \mapsto c_\lambda$, is equivalent to its complex conjugate, whence Lemma \ref{lemma:funcion_values_lines} applies, and we conclude that all elements of the sequence  $\{ c_\lambda \}$ lie on a line in the complex plane passing through the origin. This contradicts the assumption that $\{ c_\lambda \} \in \ell^2_\mathcal{O}(\Lambda)$.

We continue by proving the opposite direction, i.e., $f \nsim f_\times$ implies that $\{c_\lambda \} \in \ell^2_\mathcal{O}(\Lambda)$. Since $\{c_\lambda \} \in c_{00}(\Lambda)$ we have $\{c_\lambda \} \in \ell^2(\Lambda)$ and it suffices to show that the elements of the sequence $\{ c_\lambda \}$ do not lie on a line in the complex plane passing though the origin, provided that $f \nsim f_\times$. We prove the claim by contraposition, i.e., we assume that there exists an $\alpha \in \R$ such that $c_\lambda \in e^{i\alpha} \R$ for every $\lambda \in \Lambda$. Writing $c_\lambda = e^{i\alpha}r_\lambda$ for some $r_\lambda \in \R$ it follows that $f$ and $f_\times$ satisfy
$$
f = \sum_{\lambda \in \Lambda} e^{i \alpha} r_\lambda T_{\lambda}^S \phi, \ \ f_\times = \sum_{\lambda \in \Lambda} e^{-i \alpha} r_\lambda T_{\lambda}^S \phi.
$$

Hence, $e^{2 i \alpha} f_\times = f$ and therefore $f \sim f_\times$. This concludes the proof of the first statement.

\textbf{Proof of (2).} For the second equivalence we start by showing that $\{ c_\lambda \} \in \ell^2_\mathcal{O}(\Lambda)$ implies that $f$ and $f_\times$ are not equivalent, $f \nsim f_\times$. To this end, we assume that $\{ c_\lambda \} \in \ell^2_\mathcal{O}(\Lambda)$ and $f \sim f_\times$ and show that this leads to a contradiction. By Corollary \ref{cor:besselCorollary}, both series $f = \sum_{\lambda \in \Lambda} c_\lambda \mu(S)T_\lambda\mu(S)^{-1} \phi$ and $f_\times = \sum_{\lambda \in \Lambda} \overline{c_\lambda } \mu(S)T_\lambda\mu(S)^{-1} \phi$ converge unconditionally provided that $\{c_\lambda \} \in \ell^2_\mathcal{O}(\Lambda) \subseteq \ell^2(\Lambda)$. In particular, if $f \sim f_\times$ then there exists a constant $\nu \in \T$ such that
$$
\sum_{\lambda \in \Lambda} (c_\lambda - \nu \overline{c_\lambda} )T_\lambda\mu(S)^{-1} \phi = 0.
$$
According to Theorem \ref{thm:independence}(2), the lower boundedness of $\mathfrak{p}_{\Lambda^*}[\mu(S)^{-1}\phi]$ implies that the system $\{ T_\lambda \mu(S)^{-1} : \lambda \in \Lambda \}$ is $\ell^2(\Lambda)$-independent. Hence, $c_\lambda = \nu \overline{c_\lambda}$ for every $\lambda \in \Lambda$ and Lemma \ref{lemma:funcion_values_lines} implies that the elements of the sequence $\{ c_\lambda\}$ lie on a line in the complex plane passing through the origin, contradicting the assumption that $\{ c_\lambda \} \in \ell^2_\mathcal{O}(\Lambda)$.

To prove the other direction, assume that $\{ c_\lambda \} \notin \ell^2_\mathcal{O}(\Lambda)$. Then there exists a constant $\alpha \in \R$ such that $c_\lambda \in e^{i\alpha}\R$ for every $\lambda \in \Lambda$. Therefore, for every $\lambda \in \Lambda$ there exists a constant $r_\lambda \in \R$ such that $c_\lambda = e^{i\alpha} r_\lambda$. This implies that $e^{2 i \alpha} f_\times = f$ and consequently $f \sim f_\times$.
\end{proof}

\subsection{Implications for sampled STFT phase retrieval}\label{sec:implications}

Let $g \in \ltd$ be a window function, and let $\mathcal{L} \subseteq \Rtd$ be a subset of the time-frequency plane. Following the notation introduced in \cite{grohsLiehrJFAA}, we state the following definition.

\begin{definition}
Let $g \in \ltd$ and let $\mathcal{L} \subseteq \Rtd$. Then $(g,\mathcal{L})$ is called a uniqueness pair if every $f \in \ltd$ is determined up to a global phase by $|V_gf(\mathcal{L})|$, i.e.,
$$
|V_gf(z)| = |V_gh(z)| \ \forall z \in \mathcal{L} \implies f \sim h
$$
whenever $f,h \in \ltd$.
\end{definition}

For instance, if $g$ has an a.e.-nonvanishing ambiguity function, then $(g,\Rtd)$ is a uniqueness pair \cite[Theorem 4.27]{GrohsKoppensteinerRathmair}. For practical applications it is crucial to discretize the continuous STFT phase retrieval problem in order to achieve unique recovery of a signal from sampling sets $\mathcal{L} \subseteq \Rtd$ which are discrete -- most notably lattices. As outlined in the introduction, this question remains much less well-understood than its continuous analogue. Building on the results of the previous sections, we can formulate the following discretization barrier for multivariate STFT phase retrieval.

\begin{theorem}[Thm. \ref{thm:start1} in Sec. \ref{sec:contribution}]\label{thm:phaseRetrieval_discB}
Let $g \in \ltd$ be an arbitrary window function, and let $\mathcal{L} \subseteq \Rtd$. Then $(g,\mathcal{L})$ is never a uniqueness pair, provided that $\mathcal{L}$ satisfies one of the following two conditions:
\begin{enumerate}
    \item $\mathcal{L}=S(\Rd \times \Lambda)$ where $S \in \mathrm{Sp}_{2d}(\R)$ is a symplectic matrix and $\Lambda \subseteq \Rd$ is a lattice.
    \item $\mathcal{L}=L\Z^{2d}$ is a lattice with generating matrix $L \in \mathrm{GL}_{2d}(\R)$ which factors into $L=ST$ where $S \in \mathrm{Sp}_{2d}(\R)$ is symplectic and $T$ is a block-diagonal matrix of the form
    $$
T = \begin{pmatrix} A & 0 \\ 0 & B \end{pmatrix}, \ \ A,B \in \mathrm{GL}_d(\R).
$$
\end{enumerate}
\end{theorem}
\begin{proof}
\textbf{Proof of (1).}
The case $g=0$ is trivial. Therefore, assume in the following that $g \neq 0$. For the first statement consider a function $f \in \mathcal{V}_{\Lambda^*}^S(\mathcal{R}g)$ with defining sequence $\{ c_\lambda \} \in c_{00}(\Lambda^*) \cap \ell^2_\mathcal{O}(\Lambda^*)$. According to Theorem \ref{thm:equalSpectrograms}, the spectrogram of $f$ coincides with the spectrogram of $f_\times$ on the set $S(\Rd \times \Lambda)$. Moreover, Theorem \ref{thm:nonEquivalence}(1) shows that $f \nsim f_\times$. This proves the first claim.

\textbf{Proof of (2).}
Let $\mathcal{B}=B\Z^d$ be the lattice generated by $B$. According to statement (1), there exists two functions in $\ltd$ which do not agree up to a global phase, but their spectrograms agree on $S(\R^d \times \mathcal{B})$. The assertion is therefore a consequence of the inclusion
$$
\mathcal{L}=L\Z^{2d} = S(A\Z^d \times B \Z^d) \subseteq S(\R^d \times \mathcal{B}).
$$
\end{proof}

At this juncture, it is fruitful to compare the statement given in Theorem \ref{thm:phaseRetrieval_discB} to the situation where phase information is present. To this end, consider a window function $g \in \ltd$, a lattice $\mathcal{L} \subseteq \Rtd$, and the problem of reconstructing $f \in \ltd$ from the ordinary STFT samples
$$
V_gf(\mathcal{L}) = \{ V_gf(z) : z \in \mathcal{L} \}.
$$
Classical results in time-frequency analysis state that mild conditions on $g$ and a density assumption on $\mathcal{L}$ imply that the Gabor system $\mathcal{G}(g,\mathcal{L})$ is a frame for $\ltd$, i.e., there exist constants $0<A,B<\infty$ such that
$$
A\| f \|_2^2 \leq \sum_{\substack{\lambda \in \Lambda \\ \lambda = (x,\omega)}} |\langle f , M_{\omega}T_x g \rangle|^2 \leq B \| f \|_2^2
$$
for all $f \in \ltd$. Note that the frame property is a significantly stronger property than the uniqueness property since it gives additional stability guarantees. In particular, if $\mathcal{G}(g,\mathcal{L})$ is a frame for $\ltd$ then every $f \in \ltd$ is uniquely determined by $V_gf(\mathcal{L})$. Consider, for example, the following result on Gabor frames \cite[Theorem 6.5.2]{Groechenig}.

\begin{theorem}[Walnut]
Let $g \in W(\Rd)$ with
    $$
    W(\Rd) = \left \{ h \in L^\infty(\Rd) : \sum_{n \in \Z^d} \esssup_{x \in [0,1]^d} |h(x+n)| < \infty \right \}
    $$
    denoting the Wiener-amalgam space. If $\alpha,a,b>0$ are such that 
    $$
    a \leq \sum_{k \in \Z^d} |g(x-\alpha k)|^2 \leq b 
    $$
    for almost every $x \in \Rd$ then there exists a constant $\beta_0 = \beta_0(\alpha) > 0$ such that $\mathcal{G}(g,\alpha\Z^d \times \beta \Z^d)$ is a Gabor frame for every $0<\beta\leq \beta_0$.
\end{theorem}

Additionally, Lyubarskii, Seip, and Wallstén characterized all Gabor frames $\mathcal{G}(\varphi, \alpha \Z \times \beta \Z)$ for the Gaussian window $\varphi(t) = 2^{1/4} e^{-\pi t^2}$ in terms of a density assumption on $\alpha \Z \times \beta \Z$, namely, $\mathcal{G}(\varphi, \alpha \Z \times \beta \Z)$ is a frame if and only if $\alpha \beta <1$ \cite[Theorem 7.5.3]{Groechenig}. Moreover, Bekka showed that for every lattice $\mathcal{L} = L\Z^{2d} \subseteq \Rtd$ there exists a window function $g \in \ltd$ such that $\mathcal{G}(g,\mathcal{L})$ is a frame for $\ltd$, provided that $|\det L| \leq 1$ \cite[Theorem 11]{Heil2007}. Finally, it is well-known that if $\mathcal{L}=\alpha S \Z^{2d}$ is a symplectic lattice, $S \in \mathrm{Sp}_{2d}(\R), \alpha >0$, and if $g \in \ltd$ then $\mathcal{G}(g,\mathcal{L})$ is a frame for $\ltd$ if and only if $\mathcal{G}(h,\alpha \Z^{2d})$ is a frame for $\ltd$, where $h = \mu(S)^{-1}g$ \cite[Proposition 5]{Heil2007}.

The absence of phase information reveals a significant contrast: if $\mathcal{L} = ST\Z^{2d}$ with $S \in \mathrm{Sp}_{2d}(\R)$ and if $T$ is a block-diagonal matrix of the form
$$
T = \begin{pmatrix} A & 0 \\ 0 & B \end{pmatrix}, \ \ A,B \in \mathrm{GL}_d(\R).
$$
then according to Theorem \ref{thm:phaseRetrieval_discB}, unique recovery from phaseless STFT samples at $\mathcal{L}$ is never possible, no matter how the window function is chosen and not matter how dense the lattice is chosen. In particular, this holds for arbitrary separable lattices of the form $\mathcal{L}=\alpha \Z^d \times \beta \Z^d$ ($\alpha,\beta>0$ arbitrary density parameters) or symplectic lattices $\mathcal{L}=\alpha S \Z^{2d}$ ($S \in \mathrm{Sp}_{2d}(\R), \alpha >0$ arbitrary). We summarize the previous observations in a separate statement.

\begin{corollary}[Cor. \ref{thm:start2} in Sec. \ref{sec:contribution}]\label{cor_rem}
Let $g \in \ltd$ be an arbitrary window function, and let $\mathcal{L} \subseteq \Rtd$. Then $(g,\mathcal{L})$ is never a uniqueness pair, provided that $\mathcal{L}$ is a lattice of the following form:
\begin{enumerate}
    \item $\mathcal{L} = \alpha S\Z^{2d}$ with $S \in \mathrm{Sp}_{2d}(\R)$ a symplectic matrix and $\alpha>0$ an arbitrary density parameter.
    \item $\mathcal{L}$ is rectangular, i.e. it is generated by an invertible diagonal matrix.
    \item $\mathcal{L}$ is separable.
    \item $\mathcal{L}=\alpha \Z^d \times \beta \Z^d$ with arbitrary density parameters $\alpha,\beta \in \R \setminus \{ 0 \}$.
\end{enumerate}
\end{corollary}
\begin{proof}
Let $L \in \mathrm{GL}_{2d}(\R)$ be the generating matrix of $\mathcal{L}$. All of the assertions follow from Theorem \ref{thm:phaseRetrieval_discB} via a suitable decomposition of $L$. If $\mathcal{L}=\alpha S \Z^{2d}, S \in \mathrm{Sp}_{2d}(\R), \alpha > 0$, is a symplectic lattice, then $\mathcal{L}$ is generated by $ST$ with $T= \alpha I_{2d}$. If $\mathcal{L}$ is rectangular, then $\mathcal{L}$ is generated by $ST$ with $S=I_{2d}$ and $T=\mathrm{diag}(\alpha_1, \dots, \alpha_{2d})$, $\alpha_j \in \R \setminus \{ 0 \}, j \in \{1, \dots, 2d \}$. If $\mathcal{L}=A\Z^d \times B \Z^d, A,B \in \mathrm{GL}_d(\R)$, is a separable lattice, then $\mathcal{L}$ is generated by $ST$ with $S=I_{2d}$ and
 $$
T = \begin{pmatrix} A & 0 \\ 0 & B \end{pmatrix}, \ \ A,B \in \mathrm{GL}_d(\R).
$$
In particular, $\alpha \Z^d \times \beta \Z^d$ is a separable lattice whenever $\alpha,\beta > 0$.
\end{proof}

\begin{remark}[Shifted lattices]

Let $g \in \ltd$ be a window function, and suppose that $\mathcal{L} = L\Z^{2d}$ is a lattice as given in Theorem \ref{thm:phaseRetrieval_discB}, i.e., $L=ST$ with $S$ symplectic and $T$ block-diagonal. According to Theorem \ref{thm:phaseRetrieval_discB} there exist $f_1,f_2 \in \ltd$ such that
\begin{equation}\label{eqqq}
    |V_g(f_1)(z)| = |V_g(f_2)(z)| \ \ \forall z \in \mathcal{L}
\end{equation}
and $f_2 \nsim f_2$. For every $p=(a,b), z=(x,\omega) \in \Rd \times \Rd$ and every $u \in \ltd$, the covariance property of the STFT (Lemma \ref{lma:stft_properties}(2)) shows that
\begin{equation}\label{covvv}
    V_gu(p+z) = V_gu(a+x,b+\omega) = e^{-2 \pi i a \cdot \omega} V_g(T_{-a}M_{-b}u)(z).
\end{equation}
Accordingly, if $f_1,f_2$ are given as above and $h_1,h_2$ are defined by
\begin{equation*}
    \begin{split}
        h_1 & \coloneqq M_b T_a f_1 \\
        h_2 & \coloneqq M_b T_a f_2
    \end{split}
\end{equation*}
then the equations \eqref{eqqq} and \eqref{covvv} imply that
$$
|V_g(h_1)(y)| = |V_g(h_2)(y)| \ \ \forall y \in p + \mathcal{L}.
$$
Further, the property of two functions being equal up to a global phase is invariant under time-frequency shifts. Hence, $f_1 \nsim f_2$ if and only if $h_1 \nsim h_2$. This shows that the conclusions of Theorem \ref{thm:phaseRetrieval_discB} hold true if the lattice $\mathcal{L}$ is replaced by a shifted lattice $p+\mathcal{L}$ with $p \in \R^{2d}$ an arbitrary vector.
\end{remark}

\begin{remark}[Restriction to more regular functions spaces]\label{rem:regular}
    Let $\mathcal{L} = \mathcal{A} \times \mathcal{B}$ be a separable lattice with $\mathcal{A},\mathcal{B} \subseteq \Rd$ and let $g \in L^2(\Rd)$ be a window function. According to Corollary \ref{cor_rem} it holds that $(g,\mathcal{L})$ is not a uniqueness pair. More precisely, the proof of Theorem \ref{thm:phaseRetrieval_discB} shows that if $f \in \mathcal{V}_{\mathcal{B}^*}^S(\mathcal{R}g) = \mathcal{V}_{\mathcal{B}^*}(\mathcal{R}g)$ ($S=I_{2d}$ since the lattice is separable) has defining sequence $\{ c_\lambda \} \in c_{00}(\mathcal{B}^*) \cap \ell^2_{\mathcal{O}}(\mathcal{B}^*)$ then $f \nsim f_\times$ but $|V_gf(\mathcal{L})| = |V_g(f_\times)(\mathcal{L})|$. Observe that $f$ and $f_\times$ are a finite linear combination of ordinary shifts of $\mathcal{R}g$. Thus, if $g$ satisfies a certain regularity property such as $g \in \ltd \cap C^n(\Rd), \, n \in \N \cup \{ \infty \}$, then $f,f_\times \in \ltd \cap C^n(\Rd)$, and the corresponding counterexamples are of the same regularity as $g$. This shows the non-existence of a critical sampling density in a restriction of the problem from $L^2(\Rd)$ to $L^2(\Rd) \cap C^n(\R)$. It is worth mentioning that in recent years the STFT phase retrieval problem was studied primarily for the Gaussian window. In this case, the counterexamples $f$ and $f_\times$ are even analytic. On the other hand, if one makes a restriction to functions which satisfy an additional shift-invariant structure, then lattice uniqueness is indeed possible. Related results were shown for Paley-Wiener spaces and Gaussian shift-invariant spaces \cite{grohsliehr1}. The recovery of a real-valued function in a shift-invariant space from unsigned samples of the function itself was studied in \cite{1Thakur2011,2Groechenig2020,3Romero2021,4Alaifari2017}.
\end{remark}

We end the present subsection with the proof of Theorem \ref{thm:complex_cone} about the size of the class of non-equivalent function pairs which produce identical spectrogram samples. To that end, recall that for a window function $g \in \ltd$ and a set $\mathcal{L} \subseteq \Rtd$ the set $\mathcal{N}(g,\mathcal{L})$ is defined via
\begin{equation*}
    \begin{split}
        & \mathcal{N}(g,\mathcal{L}) \\
        & \coloneqq \left \{ f \in \ltd : \exists h \in \ltd \ \text{s.t.} \ h \nsim f \ \text{and} \ |V_gf(\mathcal{L})| = |V_gh(\mathcal{L})| \right \}.
    \end{split}
\end{equation*}
Further, recall that a set $C \subseteq V$ of a real or complex vector space $V$ is called a cone if
$$
\kappa C \subseteq C
$$
for every $\kappa > 0$ and $C$ is said to be infinite-dimensional if it is not contained in any finite-dimensional subspace of $V$. The following theorem states that the class of non-equivalent functions which produce identical phaseless STFT samples contains an infinite-dimensional cone.

\begin{theorem}[Thm. \ref{thm:complex_cone} in Sec. \ref{sec:contribution}]
Let $g \in \ltd$ be an arbitrary window function. If $\mathcal{L} \subseteq \Rtd$ satisfies $\mathcal{L} \subseteq S(\Rd \times \Lambda)$ for some symplectic matrix $S \in \mathrm{Sp}_{2d}(\R)$ and some lattice $\Lambda \subseteq \Rd$ then the set
$$
C \coloneqq \left \{ f \in \mathcal{V}_{\Lambda^*}^S(\mathcal{R}g) : f = \sum_{\lambda \in \Lambda^*} c_\lambda T_\lambda^S\mathcal{R}g, \ \{ c_\lambda \} \in c_{00}(\Lambda^*) \cap \ell^2_\mathcal{O}(\Lambda^*) \right \}
$$
is an infinite-dimensional cone which is contained in $\mathcal{N}(g,\mathcal{L})$.
\end{theorem}
\begin{proof}
If $\{ c_\lambda \} \in c_{00}(\Lambda^*) \cap \ell^2_\mathcal{O}(\Lambda^*)$ then for every $\kappa > 0$ we have $\{ \kappa c_\lambda \} \in c_{00}(\Lambda^*) \cap \ell^2_\mathcal{O}(\Lambda^*)$ which shows that $c_{00}(\Lambda^*) \cap \ell^2_\mathcal{O}(\Lambda^*)$ is a cone in $\ell^2(\Lambda^*)$. Moreover, this cone is infinite-dimensional. These observations readily imply that the set $C$ as defined above is an infinite-dimensional cone in $\ltd$. Now suppose that $f \in C$ has defining sequence $\{ c_\lambda \} \in c_{00}(\Lambda^*) \cap \ell^2_\mathcal{O}(\Lambda^*)$. Since both $\ell^2_\mathcal{O}(\Lambda^*)$ and $c_{00}(\Lambda^*)$ are invariant under complex conjugation, we have $\{ \overline{c_\lambda} \} \in c_{00}(\Lambda^*) \cap \ell^2_\mathcal{O}(\Lambda^*)$ which gives $f_\times \in C$. But $f_\times \nsim f$ by Theorem \ref{thm:nonEquivalence}(1), and $|V_gf(\mathcal{L}| = |V_g(f_\times)(\mathcal{L})|$ by Theorem \ref{thm:equalSpectrograms}. This shows that $C \subseteq \mathcal{N}(g,\mathcal{L})$ which yields the assertion of the theorem.
\end{proof}

\subsection{Examples and visualizations}\label{sec:examples_v}

In this section, we shall specify $g$ to be the Gaussian
$$
g : \R^2 \to \R, \ \  g(x,y)=e^{-x^2-y^2},
$$
which fixes the dimension to $d=2$. Let $\Lambda \subseteq \R^d$ be a lattice and let $S \in \mathrm{Sp}_{2d}(\R)$ be the identity matrix $S=I_{2d}$ so that $\mu(S)$ is the identity operator on $\ltd$ and $T_\lambda^S=T_\lambda$ is the ordinary shift operator. For $N \in \N$ let $\lambda_1, \dots, \lambda_N \in \Lambda$ be distinct points on the lattice $\Lambda$ and let $c_{\lambda_1}, \dots, c_{\lambda_N} \in \C$.
Define the sequence $\{ c_\lambda \} \in c_{00}(\Lambda)$ by
\begin{equation*}
    c_\lambda = \begin{cases}
c_{\lambda_j} ,& \lambda = \lambda_j, \, j \in \{1,\dots,N\} \\
0 ,& \text{else}
\end{cases}
\end{equation*}
and let the function $f \in \mathcal{V}^S_\Lambda(\mathcal{R}g)=\mathcal{V}_\Lambda(\mathcal{R}g)$ be defined via
$$
f = \sum_{j=1}^N c_{\lambda_j} T_{\lambda_j} \mathcal{R}g.
$$
According to Theorems \ref{thm:equalSpectrograms} and \ref{thm:nonEquivalence}, $|V_gf|$ agrees with $|V_g(f_\times)|$ on $\R^d \times \Lambda^*$ and $f \nsim f_\times$ provided that the points $c_{\lambda_1},\dots,c_{\lambda_N}$ do not lie on a line in the complex plane passing through the origin. Under these assumptions, the map $Q_xf$ defined by
\begin{equation}
    \begin{split}
        & Q_xf : \R^2 \to \R_{\geq 0}, \\
        & Q_xf(\omega_1,\omega_2)=\left ||V_gf(x,(\omega_1,\omega_2))|^2-|V_g(f_\times)(x,(\omega_1,\omega_2))|^2 \right |
    \end{split}
\end{equation}
vanishes on $\Lambda^*$ for every fixed $x \in \Rd$. In the following we visualize the function $Q_xf$ for different choices of $x,\Lambda,N,\lambda_j,c_j$.

\subsubsection{Example I}

Let $\Lambda$ be the scaled standard rectangular lattice $\Lambda=8\Z^2$ and define $\lambda_1,\lambda_2,\lambda_3 \in \Lambda$ and $c_{\lambda_1},c_{\lambda_2},c_{\lambda_3} \in \C$ via
\begin{align*}
    \lambda_1 &=(0,0), & \lambda_2&=(1,0), & \lambda_3&=(0,1), \\
    c_{\lambda_1} &= 1, & c_{\lambda_2} &= i, & c_{\lambda_3} &= 1+i.
\end{align*}
The set of points $\{ 1,i,1+i \}$ does not lie on a line in the complex plane passing through the origin, which implies that $\{ c_\lambda \} \in c_{00}(\Lambda) \cap \ell_\mathcal{O}^2(\Lambda)$. Consequently, if
$$
f_1 = c_1 T_{\lambda_1} \mathcal{R}g+c_2 T_{\lambda_2} \mathcal{R}g+c_3 T_{\lambda_3} \mathcal{R}g
$$
then according to Theorem \ref{thm:equalSpectrograms}, and Theorem \ref{thm:nonEquivalence}, $Q_xf_1$ vanishes on $\Lambda^* = \tfrac{1}{8}\Z^2$ and $f_1 \nsim (f_1)_\times$. Contour plots of $Q_xf_1$ for different choices of $x$ are provided in Figure \ref{fig:rectangular}.

\begin{figure}
\centering
\begin{subfigure}{.5\textwidth}
  \centering
  \includegraphics[width=0.9\linewidth]{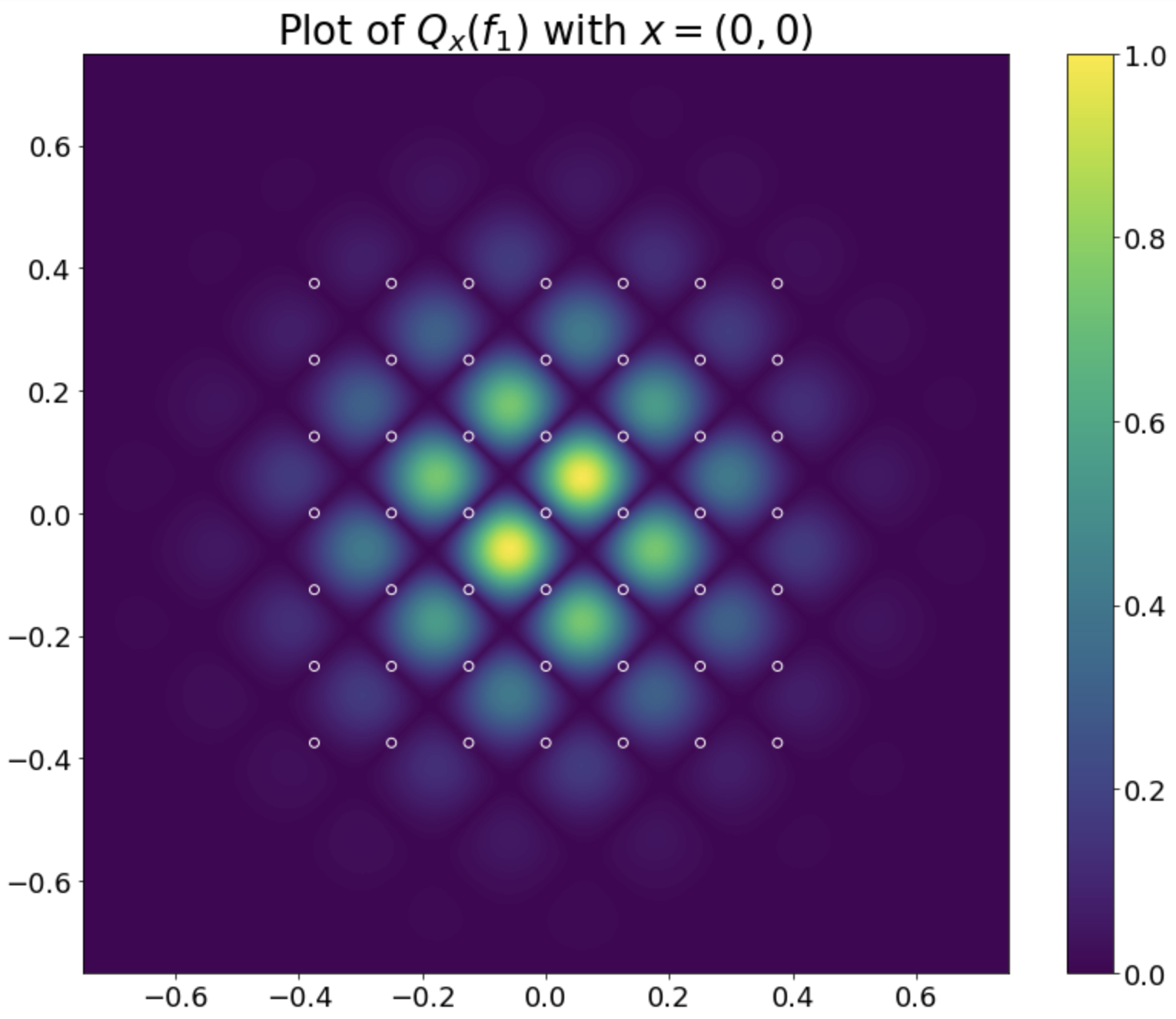}
  \caption{}
  \label{fig:sub1}
\end{subfigure}%
\begin{subfigure}{.5\textwidth}
  \centering
  \includegraphics[width=0.9\linewidth]{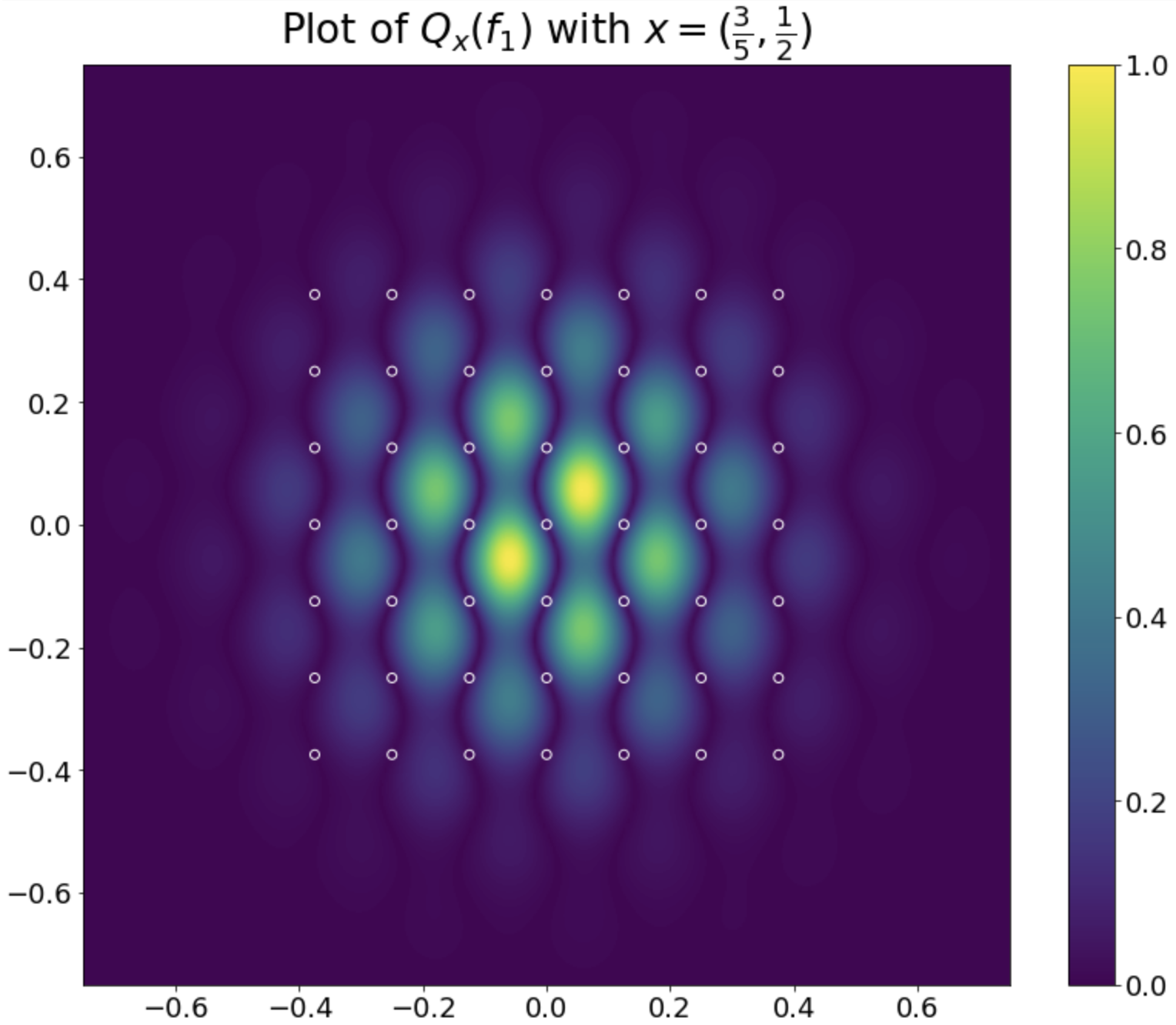}
  \caption{}
  \label{fig:sub2}
\end{subfigure}
\caption{Contour plots of the function $Q_xf_1$ from Example I for different choices of $x$. The white dots represent the lattice $\tfrac{1}{8}\Z^2$. The figure shows that $Q_xf_1$ vanishes on $\tfrac{1}{8}\Z^2$.}
\label{fig:rectangular}
\end{figure}

\subsubsection{Example II} 

Suppose that $\Lambda$ is generated by the matrix $B \in \mathrm{GL}_2(\R)$ defined by
$$
B = 5 \begin{pmatrix} 1 & 0 \\ -\tfrac{1}{\sqrt{3}} & \tfrac{2}{\sqrt{3}} \end{pmatrix} = (a,b), \ \ a = \begin{pmatrix} 5  \\ -\tfrac{5}{\sqrt{3}} \end{pmatrix}, \ \ b=\begin{pmatrix}  0 \\  \tfrac{10}{\sqrt{3}} \end{pmatrix}.
$$
Further, define $\lambda_1,\lambda_2,\lambda_3,\lambda_4 \in \Lambda$ and $c_{\lambda_1},c_{\lambda_2},c_{\lambda_3},c_{\lambda_4} \in \C$ via
\begin{align*}
    \lambda_1 &=(0,0), & \lambda_2&=a, & \lambda_3&=b & \lambda_4&=a+b, \\
    c_{\lambda_1} &= 1, & c_{\lambda_2} &= i, & c_{\lambda_3} &= 1+i & c_{\lambda_4}&=\tfrac{1}{2}+\tfrac{1}{2}i
\end{align*}
and let
\begin{equation*}
    \begin{split}
        f_2 & = c_{\lambda_1} T_{\lambda_1} \mathcal{R}g+c_{\lambda_2} T_{\lambda_2} \mathcal{R}g+c_{\lambda_3} T_{\lambda_3} \mathcal{R}g, \\
        f_3 & = c_{\lambda_1}T_{\lambda_1} \mathcal{R}g+c_{\lambda_2} T_{\lambda_2} \mathcal{R}g+ c_{\lambda_3} T_{\lambda_3} \mathcal{R}g+c_{\lambda_4} T_{\lambda_4} \mathcal{R}g.
    \end{split}
\end{equation*}
Contour plots of $Q_xf_2$ and $Q_xf_3$ are provided in Figure \ref{fig:hexagonal}. Note that both $Q_xf_2$ and $Q_xf_3$ vanish on $\Lambda^*$ which is the lattice generated by the matrix $B^{-T}$,
$$
B^{-T} = \frac{1}{5} \begin{pmatrix} 1 & \tfrac{1}{2} \\ 0 & \tfrac{\sqrt{3}}{2} \end{pmatrix}.
$$
This is the generating matrix of a scaled hexagonal lattice \cite[p. 5]{conway_sphere}.

\begin{figure}
\centering
\begin{subfigure}{.5\textwidth}
  \centering
  \includegraphics[width=0.9\linewidth]{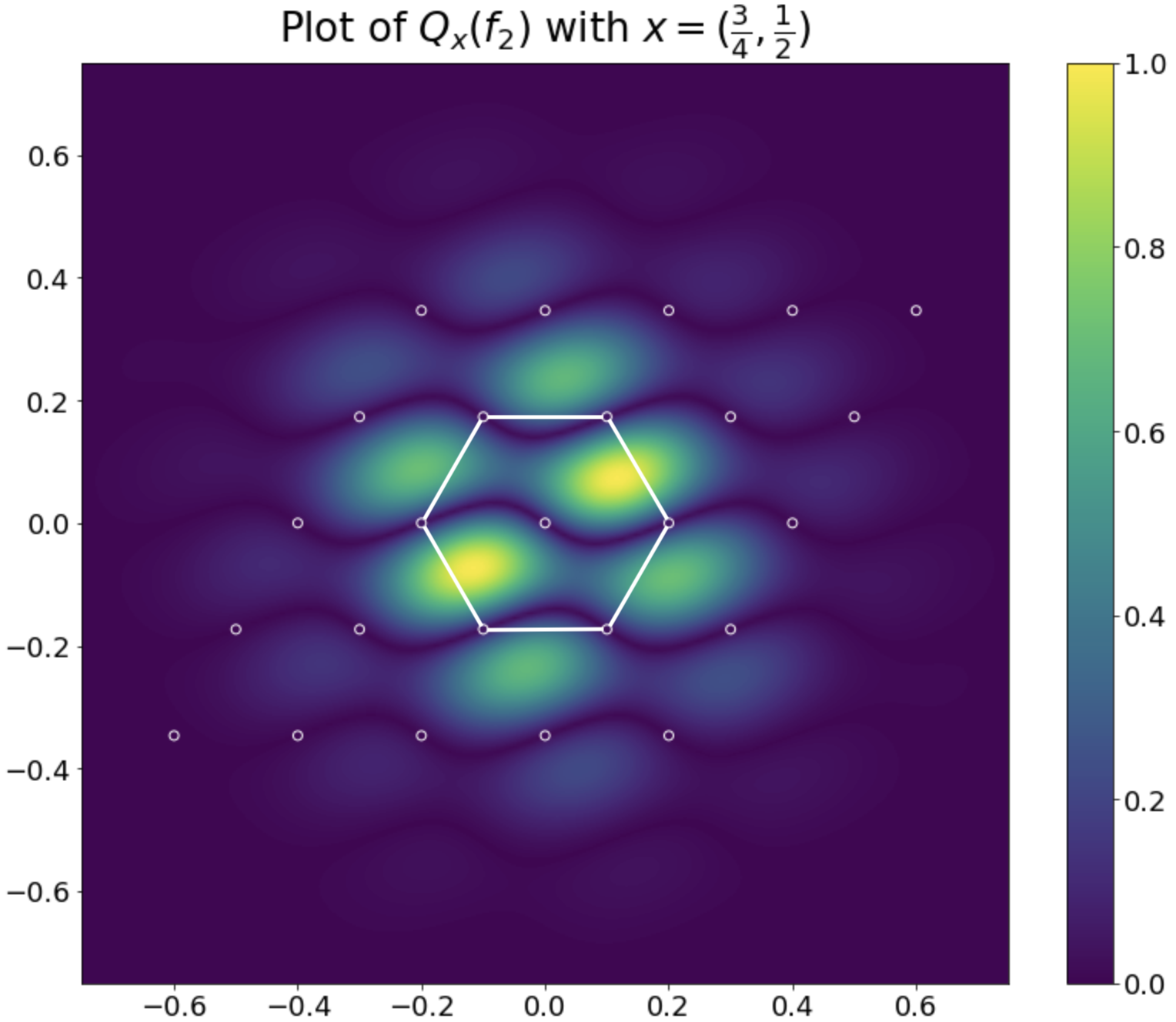}
  \caption{}
  \label{fig:sub3}
\end{subfigure}%
\begin{subfigure}{.5\textwidth}
  \centering
  \includegraphics[width=0.9\linewidth]{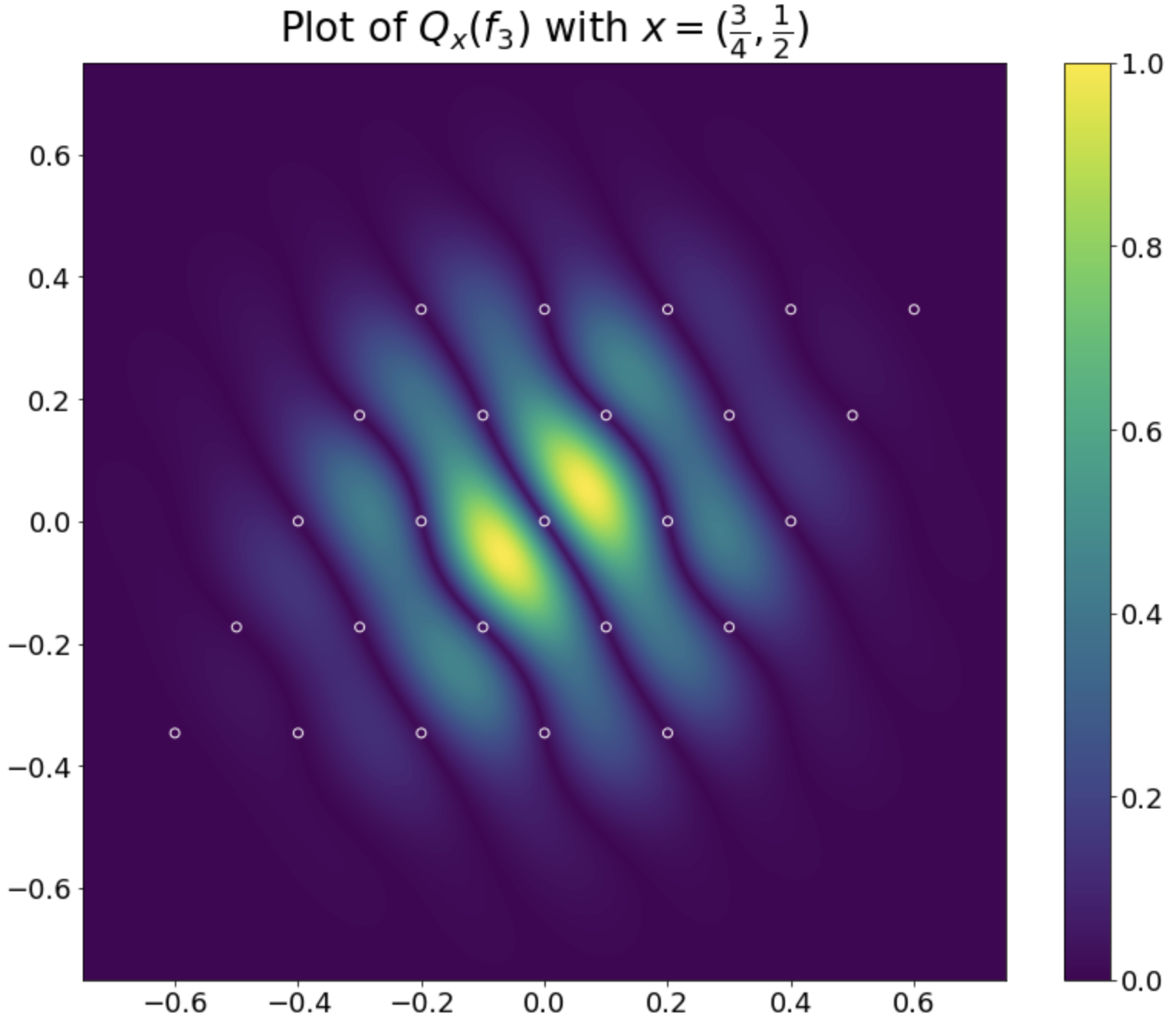}
  \caption{}
  \label{fig:sub4}
\end{subfigure}
\caption{Contour plots of the functions $Q_xf_2$ and $Q_xf_3$ from example II. The white dots represent lattice points of a scaled hexagonal lattice. The figure shows that both $Q_xf_2$ and $Q_xf_3$ vanish on this lattice.}
\label{fig:hexagonal}
\end{figure}

\subsection{Separable lattices and real-valued window functions}\label{sec:separability}

In the previous sections, we considered the STFT phase retrieval problem for arbitrary window functions. We now pay particular attention to the setting where the window function is real-valued and the spectrogram is sampled on a separable lattice. These are the most usual assumptions made in time-frequency analysis and its applications. One of these applications, ptychography, was discussed in detail in Section \ref{sec:ptychography}.

\subsubsection{Pauli-type non-uniqueness}

Recall that the Pauli problem is the question of whether a function $f \in \lt$ is determined up to a global phase by its modulus $|f|$ and the modulus of its Fourier transform $|\ft f|$. It is well-known that this is, in general, not the case, i.e., one can construct infinitely many so-called Pauli partners $f_1,f_2 \in \lt$ such that $|f_1|=|f_2|$ and $|\ft f_1| = |\ft f_2|$ \cite{corbett_hurst_1977, VOGT1978365}. A natural variation of the Pauli problem replaces the Fourier transform $\ft$ by a different operator. In Section \ref{sec:implications} we saw that for every $g \in L^2(\Rd,\R)$ and every separable lattice $\mathcal{L}\subseteq \Rtd$ there exists $f_1, f_2 \in \ltd$ such that $|V_g(f_1)(\mathcal{L})|=|V_g(f_2)(\mathcal{L})|$ and $f_1 \nsim f_2$. Consider the Pauli-type problem where, in addition to identical spectrogram samples, one also has $|f_1|=|f_2|$. The next result demonstrates that the induced Pauli problem still fails to be unique: additional knowledge of the moduli of two functions, does not improve the uniqueness property. Its proof highlights how the results derived in the previous sections serve as handy machinery for providing counterexamples in other contexts appearing frequently in phase retrieval.

\begin{theorem}[Thm. \ref{thm:start3} in Sec. \ref{sec:contribution}]
Let $g \in L^2(\Rd,\R)$ be a real-valued window function, and let $\mathcal{L} \subseteq \Rtd$ be a separable lattice. Then there exist $f_1,f_2 \in \ltd$ such that
\begin{enumerate}
    \item $|V_g(f_1)(z)|=|V_g(f_2)(z)|$ for all $z \in \mathcal{L}$
    \item $|f_1(t)| = |f_2(t)|$ for all $t \in \Rd$
    \item $f_1 \nsim f_2$.
\end{enumerate}
\end{theorem}
\begin{proof}
Suppose that $\mathcal{L} = \mathcal{A} \times \mathcal{B}$ such that $\mathcal{A}, \mathcal{B}$ are lattices in $\Rd$ generated by the matrices $A,B \in \mathrm{GL}_d(\R)$, respectively. For a non-zero element $0 \neq \lambda \in \mathcal{B}^*$ of the reciprocal lattice of $\mathcal{B}$, define
$$
f_1 \coloneqq T_{-\lambda} \mathcal{R}g + i T_\lambda \mathcal{R}g \in \mathcal{V}_{\mathcal{B}^*}(\mathcal{R}g).
$$
Further, let $f_2 \coloneqq (f_1)_\times$ be the corresponding function which has a complex-conjugate defining sequence, i.e.,
$$
f_2 \coloneqq T_{-\lambda} \mathcal{R}g - i T_\lambda \mathcal{R}g.
$$
Since $\mathcal{V}_{\mathcal{B}^*}(\mathcal{R}g) = \mathcal{V}_{\mathcal{B}^*}^S(\mathcal{R}g)$ with $S=I_{2d}$, Theorem \ref{thm:equalSpectrograms} shows that
$$
|V_g(f_1)(z)|=|V_g(f_2)(z)|
$$
for every $z \in \Rd \times (\mathcal{B}^*)^* = \Rd \times \mathcal{B} \supseteq \mathcal{A} \times \mathcal{B}$ which yields property (1). Since $g$ is real-valued, we have $f_1 = \overline{f_2}$.
Thus, $|f_1|=|f_2|$ which shows that property (2) is satisfied. Property (3) follows by selection of $f_1$: we have
$$
f_1 \coloneqq c_{-\lambda} T_{-\lambda} \mathcal{R}g + c_\lambda T_\lambda \mathcal{R}g
$$
with $c_{-\lambda}=1$ and $c_\lambda=i$. The set $\{ 1,i \}$ does not lie on a line in the complex plane passing through the origin. Hence, the defining sequence $\{ c_\lambda \} \subseteq \C$ of $f_1$ satisfies $\{ c_\lambda \} \in c_{00}(\mathcal{B}^*) \cap \ell^2_\mathcal{O}(\mathcal{B}^*)$. In particular, Theorem \ref{thm:nonEquivalence} shows that $f_1 \nsim (f_1)_\times = f_2$ and this concludes the proof of the statement.
\end{proof}

\subsubsection{Restriction to the space $L^2(\Rd, \R)$}

Next, we investigate the situation where, in addition to a real-valued window function, one has the prior knowledge that the underlying signal space consists of real-valued functions, i.e., we assume that all considered input functions belong to the space $L^2(\Rd,\R)$ of real-valued, square-integrable functions. The corresponding uniqueness problem asks for the validity of the implication
$$
|V_gf(z)| = |V_gh(z)| \ \ \forall z \in \mathcal{L} \implies f \sim h
$$
provided that $f,h \in L^2(\Rd,\R)$.
Notice that in this setting the condition $f \sim h$ for $f,h \in L^2(\Rd,\R)$ means that there exists a constant $\nu \in \{ -1,1 \}$ such that $f=\nu h$. In other words, $f$ is equal to $h$ up to a sign factor, and we have the equivalence
$$
f \sim h \iff ( f = h \ \ \mathrm{or} \ \ f=-h ).
$$
The corresponding phase retrieval problem is therefore also known as the sign retrieval problem. In order to study the real-valued regime, we start by introducing the space of Hermitian sequences. 

\begin{definition}
Let $\Lambda \subseteq \Rd$ be a lattice. We define the subspace of Hermitian sequences in $\ell^2(\Lambda)$ via
$$
\ell^2_\mathcal{H}(\Lambda) \coloneqq \left \{ \{ c_\lambda \} \in \ell^2(\Lambda) : c_{-\lambda} = \overline{c_{\lambda}} \right \}.
$$
\end{definition}

Suppose now that $\phi \in L^2(\Rd,\R)$ is a real-valued generating function. If $\Lambda \subseteq \Rd$ is a lattice, $S \in \mathrm{Sp}_{2d}(\R)$ is a symplectic matrix, and $f \in \mathcal{V}_\Lambda^S(\Lambda)$ has an Hermitian defining sequence, then the following statement shows how a suitable choice of $S$ implies real-valuedness of $f$.

\begin{proposition}\label{prop:r_valuedness_condition}
Let $\phi \in L^2(\Rd,\R)$ be real-valued, and let $\Lambda \subseteq \Rd$ be a lattice. Suppose that $\{c_\lambda \} \subseteq \C$ is the defining sequence of $f \in \mathcal{V}_\Lambda^{-\mathcal{J}}(\phi)$ where $\mathcal{J} \in \mathrm{Sp}_{2d}(\R)$ denotes the standard symplectic matrix. Then the following holds:
\begin{enumerate}
    \item If $\{c_\lambda \} \in c_{00}(\Lambda) \cap \ell^2_\mathcal{H}(\Lambda)$ then $f$ is real-valued.
    \item If $\mathfrak{p}_{\Lambda^*}[\ift \phi] \in L^\infty(\mathcal{P}(\Lambda^*))$ then $f$ is real-valued provided that $\{ c_\lambda \} \in \ell^2_\mathcal{H}(\Lambda)$.
\end{enumerate}
\end{proposition}
\begin{proof}
\textbf{Proof of (1).}
If $\mathcal{J}$ denotes the standard symplectic matrix, then $\mu(-\mathcal{J})=\ft$ is the Fourier transform \cite[Example 9.4.1]{Groechenig}. Hence, if $\{c_\lambda \} \in c_{00}(\Lambda)$ is the defining sequence of $f$ then
\begin{equation}\label{rser}
    f = \sum_{\lambda \in \Lambda} c_\lambda \ft T_\lambda \ft^{-1} \phi = \sum_{\lambda \in \Lambda} c_\lambda M_{-\lambda} \phi.
\end{equation}
Consider a partition of the lattice $\Lambda$ of the form
$$
\Lambda = \{ 0 \} \cup I_+ \cup I_- \ \ \text{s.t.} \ \  -I_+ = I_-.
$$
Using this partition, the map $f$ can be written as
\begin{equation}\label{eq:real-val}
    \begin{split}
        f(t) &= c_0 \phi(t) + \sum_{\lambda \in I_+} c_\lambda M_{-\lambda} \phi + \sum_{\lambda \in I_-} c_\lambda M_{-\lambda} \phi(t) \\
        & = c_0 \phi(t) +  \sum_{\lambda \in I_+} c_\lambda M_{-\lambda} \phi + \sum_{\lambda \in I_+} c_{-\lambda} M_{\lambda} \phi(t) \\
        & = c_0 \phi(t) + \sum_{\lambda \in I_+} \left ( c_\lambda e^{-2\pi i \lambda t} + \overline{c_\lambda e^{-2\pi i \lambda t}} \right ) \phi(t) \\
        & = c_0 \phi(t) + \sum_{\lambda \in I_+} 2 \re \left (  c_\lambda e^{-2\pi i \lambda t} \right ) \phi(t),
    \end{split}
\end{equation}
where we use the fact that $\{ c_\lambda \}$ belongs to the set $\ell^2_\mathcal{H}(\Lambda)$. This property of $\{ c_\lambda \}$ further implies that $c_0 \in \R$. Since $\phi$ is real-valued, the identity derived in equation \eqref{eq:real-val} shows that $f$ is real-valued.

\textbf{Proof of (2).}
To prove the second statement we observe that if $\mathfrak{p}_{\Lambda^*}[\ift u] \in L^\infty(\mathcal{P}(\Lambda^*))$ then by Proposition \ref{prop:bessel_condition} the series given in equation \eqref{rser} converges unconditionally. With an argument analogous to the previous case we obtain the relation
$$
f(t) = c_0 \phi(t) + \sum_{\lambda \in I_+} 2 \re \left (  c_\lambda e^{-2\pi i \lambda t} \right ) \phi(t)
$$
with a suitable choice of $I_+ \subseteq \Lambda$. Thus, $f$ is real-valued.
\end{proof}

The previous proposition leads to the following statement on STFT phase retrieval in a real-valued regime.

\begin{theorem}[Thm. \ref{thm:intro:sign_retrieval} in Sec. \ref{sec:contribution}]\label{thm:r_valued_counterexamples}
Let $g \in L^2(\Rd,\R)$ be a real-valued window function, and suppose that $\mathcal{L} \subseteq \Rd$ satisfies one of the following conditions:
\begin{enumerate}
    \item $\mathcal{L} = \Lambda \times \R^d$ with $\Lambda$ a lattice in $\Rd$
    \item $\mathcal{L}$ is a separable lattice of the form $\mathcal{L} = \mathcal{A} \times \mathcal{B}$ where $\mathcal{A}$ and $\mathcal{B}$ are lattices in $\Rd$.
\end{enumerate}
Then there exists two real-valued functions $f_1,f_2 \in L^2(\Rd,\R)$ such that
$$
|V_g(f_1)(z)| = |V_g(f_2)(z)| \ \forall z \in \mathcal{L} \ \ \text{and} \ \ f_1 \nsim f_2.
$$
\end{theorem}
\begin{proof}
Let $\Lambda \subseteq \Rd$ be a lattice, and suppose that $f_1 \in \mathcal{V}_{\Lambda^*}^{-\mathcal{J}}(\mathcal{R}g)$ has defining sequence $\{ c_\lambda \}$ belonging to the non-empty intersection $c_{00}(\Lambda^*) \cap \ell^2_\mathcal{O}(\Lambda^*) \cap \ell^2_\mathcal{H}(\Lambda^*)$. Set $f_2 \coloneqq (f_1)_\times$. According to Theorem \ref{thm:nonEquivalence} we have $f_1 \nsim f_2$. In addition, Proposition \ref{prop:r_valuedness_condition} shows that $f_1$ is real-valued. Noting that both $c_{00}(\Lambda^*)$ and $\ell^2_\mathcal{H}(\Lambda^*)$ are invariant under complex conjugation shows that $f_2$ is real-valued as well. Finally, Theorem \ref{thm:equalSpectrograms} implies that
$$
|V_g(f_1)(z)| = |V_g(f_2)(z)|
$$
for every $z \in -\mathcal{J}(\Rd \times \Lambda) = \Lambda \times (-\Rd) = \Lambda \times \Rd$. The second part of the statement follows from the property that every separable lattice $\mathcal{A} \times \mathcal{B}$ is contained in a set of the form $ \Lambda \times \Rd$.
\end{proof}

Theorem \ref{thm:r_valued_counterexamples} proves the existence of two non-equivalent, real-valued functions $f_1,f_2 \in L^2(\Rd, \R)$ for which their spectrograms agree on a separable lattice $\mathcal{L} = \mathcal{A} \times \mathcal{B}$. Clearly, the conclusions of Theorem \ref{thm:r_valued_counterexamples} hold true for every lattice $\mathcal{D}$ which is contained in a separable lattice $\mathcal{L}$, $\mathcal{D} \subseteq \mathcal{L}$. If $\mathcal{L}$ is generated by $L$ and $\mathcal{D}$ is generated by $D$, the property of $\mathcal{D}$ being contained in $\mathcal{L}$ means that $L^{-1}D$ is an integral matrix. A class of matrices that satisfies this condition is the important class of matrices with rational entries. Such matrices play an indispensable role in the numerical treatment of the STFT phase retrieval problem.

\begin{corollary}[Cor. \ref{thm:start5} in Sec. \ref{sec:contribution}]
Suppose that $\mathcal{L} \subseteq \R^{2d}$ is a lattice which is generated by an invertible matrix $L$ with rational entries, $L \in \mathrm{GL}_{2d}(\Q)$. Then for every real-valued window function $g \in L^2(\Rd,\R)$ there exists two real-values functions $f_1,f_2 \in L^2(\Rd,\R)$ such that
\begin{enumerate}
    \item $|V_g(f_1)(z)| = |V_g(f_2)(z)|$ for every $z \in \mathcal{L}$
    \item $f_1 \nsim f_2$.
\end{enumerate}
\end{corollary}

\begin{proof}
Let $q_i = (q_{i,1}, \dots, q_{i,2d} ) \in \Q^{2d}, i \in \{ 1, \dots, 2d \}$, be the row vectors of $L$. Since every $q_{ij}$ is rational, we can write
$
q_{ij} = \frac{a_{ij}}{b_{ij}}, a_{ij} \in \Z, b_{ij} \in \Z \setminus \{ 0 \}.
$
As a consequence, we have
$$
q_i = \frac{1}{\prod_{j=1}^{2d} b_{ij}} (c_{i1}, \dots, c_{i,2d}), \ \ c_{ij} = a_{ij} \prod_{k=1, k \neq j}^{2d} b_{ik}.
$$
Now let $c_i \coloneqq (c_{i1}, \dots, c_{i,2d})$ and define
$$
D \coloneqq \mathrm{diag} \left ( \frac{1}{\prod_{j=1}^{2d} b_{1j}}, \frac{1}{\prod_{j=1}^{2d} b_{2j}}, \dots, \frac{1}{\prod_{j=1}^{2d} b_{2d,j}}  \right ).
$$
If $z \in \Z^{2d}$ then in view of the notations above the product $Lz$ satisfies
\begin{equation}\label{eq:LD}
    Lz = \begin{pmatrix}  q_1 \cdot z  \\ \vdots \\ q_{2d} \cdot z \end{pmatrix} = \begin{pmatrix} \frac{1}{\prod_{j=1}^{2d} b_{1j}} c_1 \cdot z \\ \vdots \\ \frac{1}{\prod_{j=1}^{2d} b_{2d,j}} c_{2d} \cdot z \end{pmatrix} = D \begin{pmatrix} c_1 \cdot z \\ \vdots \\ c_{2d} \cdot z \end{pmatrix}.
\end{equation}
Since $c_j \cdot z \in \Z$ for every $j \in \{ 1, \dots, 2d \}$ it follows from equation \eqref{eq:LD} that $\mathcal L \subseteq \mathcal{D}$ where $\mathcal{D}$ is the lattice generated by the diagonal matrix $D$. Since $\mathcal{D}$ is separable, Theorem \ref{thm:r_valued_counterexamples} implies the existence of two real-valued functions $f_1, f_2 \in L^2(\Rd,\R)$ such that $|V_g(f_1)(w)| = |V_g(f_2)(w)|$ for every $w \in \mathcal{D}$ and, in addition, $f_1 \nsim f_2$. The statement follows from the inclusion $\mathcal{L} \subseteq \mathcal{D}$.
\end{proof}

We finalize the present subsection with the proof of Theorem \ref{thm:real_cone} about the size of the class of real-valued, non-equivalent functions in $\ltd$ which produce identical phaseless STFT samples on sets of the form $\mathcal{L} = \Lambda \times \Rd$ with $\Lambda \subseteq \Rd$ a lattice. Recall that the set $\mathcal{N}_\R(g,\mathcal{L})$ is defined by
\begin{equation*}
    \begin{split}
        & \mathcal{N}_\R(g,\mathcal{L})  \\ & \coloneqq \left \{ f \in L^2(\Rd, \R) :  \exists h \in L^2(\Rd, \R) \ \text{s.t.} \ h \nsim f \ \text{and} \ |V_gf(\mathcal{L})| = |V_gh(\mathcal{L})| \right \}.
    \end{split}
\end{equation*}

\begin{theorem}[Thm. \ref{thm:real_cone} in Sec. \ref{sec:contribution}]
Let $g \in \ltd$ be an arbitrary real-valued window function. If $\mathcal{L} \subseteq \Rtd$ satisfies $\mathcal{L} \subseteq \Lambda \times \Rd$ for some lattice $\Lambda \subseteq \Rd$ then the set
$$
C \coloneqq \left \{ f \in \mathcal{V}_{\Lambda^*}^{-\mathcal{J}}(\mathcal{R}g) : f = \sum_{\lambda \in \Lambda^*} c_\lambda T_\lambda^{-\mathcal{J}}\mathcal{R}g, \ \{ c_\lambda \} \in c_{00}(\Lambda^*) \cap \ell^2_\mathcal{O}(\Lambda^*) \cap \ell^2_\mathcal{H}(\Lambda^*) \right \}
$$
is an infinite-dimensional cone which is contained in $\mathcal{N}_\R(g,\mathcal{L})$.
\end{theorem}
\begin{proof}
If $\{ c_\lambda \} \in c_{00}(\Lambda^*) \cap \ell^2_\mathcal{O}(\Lambda^*) \cap \ell^2_\mathcal{H}(\Lambda^*)$ then for every $\kappa > 0$ we have $\{ \kappa c_\lambda \} \in c_{00}(\Lambda^*) \cap \ell^2_\mathcal{O}(\Lambda^*) \cap \ell^2_\mathcal{H}(\Lambda^*)$ which shows that $c_{00}(\Lambda^*) \cap \ell^2_\mathcal{O}(\Lambda^*) \cap \ell^2_\mathcal{H}(\Lambda^*)$ is a cone in $\ell^2(\Lambda^*)$. Moreover, this cone is infinite-dimensional. These observations readily imply that the set $C$ as defined above is an infinite-dimensional cone in $\ltd$. Proposition \ref{prop:r_valuedness_condition} shows that every element in $C$ is real-valued and therefore $C \subseteq L^2(\Rd,\R)$. Now suppose that $f \in C$ has defining sequence $\{ c_\lambda \} \in c_{00}(\Lambda^*) \cap \ell^2_\mathcal{O}(\Lambda^*) \cap \ell^2_\mathcal{H}(\Lambda^*)$. Since $\ell^2_\mathcal{O}(\Lambda^*), c_{00}(\Lambda^*)$ and $\ell^2_\mathcal{H}(\Lambda^*)$ are invariant under complex conjugation, we have $\{ \overline{c_\lambda} \} \in c_{00}(\Lambda^*) \cap \ell^2_\mathcal{O}(\Lambda^*) \cap \ell^2_\mathcal{H}(\Lambda^*)$ which gives $f_\times \in C$. But $f_\times \nsim f$ by Theorem \ref{thm:nonEquivalence}(1) and $|V_gf(\mathcal{L}| = |V_g(f_\times)(\mathcal{L})|$ by Theorem \ref{thm:equalSpectrograms}. This shows that $C \subseteq \mathcal{N}_\R(g,\mathcal{L})$.
\end{proof}

\subsection{The univariate case $d=1$ and Fock spaces}\label{univariate_case}

Consider the univariate STFT phase retrieval problem, i.e., the situation where the window function $g$ is an element of $\lt$ and the sampling points $\mathcal{L}$ are a subset of $\R^2$. The main result of the previous paper \cite{grohsLiehrJFAA} states that for every window function $g \in \lt$ and every lattice $\mathcal{L}=L\Z^2, L \in \mathrm{GL}_2(\R)$, there exist $f_1,f_2 \in \lt$ such that $f_1 \nsim f_2$ and $|V_gf_1(z)|=|V_gf_2(z)|$ for every $z \in \mathcal{L}$. For the special case of the centered Gaussian window $g(t)=e^{-\pi t^2}$ the result was derived in \cite{alaifari2020phase}. In both papers, the construction of the functions $f_1,f_2$ with the above properties uses ideas from the theory of fractional Fourier transforms. These are metaplectic operators, namely the operators corresponding to rotation matrices 
$$
\begin{pmatrix} \cos \theta & - \sin \theta \\ \sin \theta & \cos \theta \end{pmatrix} \in \mathrm{Sp}_2(\R), \ \ \theta \in \R.
$$
We refer the reader to an article by de Gosson and Luef for a systematic study of the correspondence between metaplectic operators and fractional Fourier transforms \cite{gossonLuef}. We now show that the statements derived in \cite{alaifari2020phase,grohsLiehrJFAA} follow as a by-product of the results derived in the present paper. This is a consequence of the fact that in $\R^2$ every lattice is symplectic.

\begin{corollary}[Main theorem of \cite{grohsLiehrJFAA}]\label{thm:1d_case}
Let $g \in \ltd$ and let $\mathcal{L} \subseteq \Rtd$ be an \emph{arbitrary} lattice. If $d=1$ then $(g,\mathcal{L})$ is not a uniqueness pair.
\end{corollary}
\begin{proof}
Let $\alpha \coloneqq \det L, S \coloneqq \alpha^{-1}L \in \mathrm{Sp}_2(\R)$ and let $\{ c_\lambda \} \subseteq \C$ be the defining sequence of $f \in \mathcal{V}_{\alpha^{-1}\Z}^S(\mathcal{R}g)$ such that $\{ c_\lambda \} \in c_{00}(\alpha^{-1}\Z) \cap \ell^2_\mathcal{O}(\alpha^{-1}\Z)$. According to Corollary \ref{cor:1d_lattice} it holds that $|V_gf(z)| = |V_g(f_\times)(z)|$ for every $z \in \mathcal{L}$ whereas $f \nsim f_\times$ by Theorem \ref{thm:nonEquivalence}. This yields the assertion.
\end{proof}

\begin{remark}[Fock spaces]

Let $z = x+i\omega \in \C, x,\omega \in \R$, and let $\mu$ be the Gaussian measure $d\mu(z) = e^{-\pi |z|^2} \, dxdy$. The Bargmann-Fock space $F^2(\C)$ is the collection of all entire functions $a : \C \to \C$ for which
$$
\| a \|_{F^2(\C)} \coloneqq \int_\C |a(z)|^2 \, d\mu(z) < \infty.
$$
If $f \in \lt$ then we can define an entire function $Bf : \C \to \C$ via
$$
Bf(z) = 2^{\frac{1}{4}} \int_\R f(t) e^{2\pi t z - \pi t^2 - \frac{\pi}{2}z^2} \, dt.
$$
The map $Bf$ is called the Bargmann transform of $f$ and it can be shown that $B$ is an isometry mapping $\lt$ onto $F^2(\C)$ \cite[Theorem 6.8]{zhu}. The Bargmann transform is closely related to the STFT with Gaussian window. If $\varphi(t) = 2^{\frac{1}{4}} e^{-\pi t^2}$ then for every $f \in \lt$ and every $z = x+i\omega \in \C, x,\omega \in \R$, we have
$$
V_\varphi f(x,-\omega) = e^{\pi i x \omega - \frac{\pi}{2}|z|^2} Bf(z).
$$
In other words, $V_\varphi f$ is equal to $Bf$ up to a reflection and up to a multiplicative non-zero weighting factor which is independent of $f$. In particular, if $\mathcal{L} \subseteq \R^2$ then according to \cite[Proposition 3.4.1]{Groechenig}, for every $f,h \in \lt$ we have
\begin{equation}\label{fock_eq}
    \begin{split}
        & |V_\varphi f(x,-\omega)| = |V_\varphi h(x,-\omega)| \ \ \forall (x,\omega) \in \mathcal{L}  \\
        \iff & |Bf(x+i\omega)| = |Bh(x+i\omega)|\ \ \forall (x,\omega) \in \mathcal{L}.
    \end{split}
\end{equation}
Classical uniqueness theory in Bargmann-Fock spaces shows that if $\mathcal{L}=L\Z^2 \subseteq \R^2 \simeq \C, L \in \mathrm{GL}_2(\R)$, is a lattice such that $\det L \leq 1$ then $\mathcal{L}$ is a uniqueness set for $F^2(\C)$, i.e., two functions $a,b \in F^2(\C)$ are identical provided that $a$ and $b$ agree on $\mathcal{L}$ \cite{BARGMANN1971221,Perelomov1971,Seip+1992+91+106}. In an analogous and frequently used terminology, this means that the system of \emph{coherent states} $\{ e^{2 \pi i \omega t} \varphi(t-x) : (x,\omega) \in \mathcal{L} \}$ is complete in $\lt$. Now observe that the Bargmann transform is a linear bijection between $\lt$ and $F^2(\C)$ and that $\mathcal{L}' \coloneqq \{ (x,-\omega) : (x,\omega) \in \mathcal{L} \}$ is a lattice provided that $\mathcal{L}$ is a lattice. Combing these observations with Corollary \ref{thm:1d_case} and the equivalence given in equation \eqref{fock_eq} shows that for every lattice $\mathcal{L} \subseteq \R^2 \simeq \C$ there exists two functions $a,b \in F^2(\C)$ such that
$$
|a(z)| = |b(z)| \ \ \forall z \in \mathcal{L} \ \ \mathrm{and} \ \ a \nsim b,
$$
in complete contrast to the setting where phase information is present.
\end{remark}

\section{Conclusion}

Motivated by important applications in physics and imaging sciences, in this article we have derived a non-uniqueness theory in sampled STFT phase retrieval which leads to the formulation of several fundamental discretization barriers. The main results highlight that the STFT phase retrieval problem fails to be unique if the samples are located on lattices. In fact, we developed an extensive machinery for the construction of non-equivalent function pairs which produce identical spectrogram samples on certain prescribed lattices. As an application, we showed that the Pauli problem, which is induced via phaseless sampling of the STFT, is not unique. Moreover, the problem even fails to be unique if the signal class is restricted to real functions. The established theorems emphasize a foundational difference of sampling without phase compared to ordinary sampling of the STFT: no matter how the window function is chosen, there exists no critical sampling density that allows uniqueness to be achieved via lattice sampling (as would be the case for the Nyquist rate in classical sampling theory). In addition, the results highlight the stark contrast to the case where the spectrogram of a function is sampled on a continuous domain (e.g., for Gaussian windows, the STFT phase retrieval problem is unique via sampling on an arbitrary open set). The proofs made use of techniques from several areas in mathematics, in particular symplectic geometry and the theory of shift-invariant spaces as well as linear independence properties of systems of translates. Finally, the article gives rise to fruitful future research: since lattice sampling does not guarantee uniqueness, one might question whether irregular sampling or a suitable increase of the redundancy of the sampling set is beneficial.

\section{Appendix}

\subsection{Proof of Proposition \ref{prop:bessel_condition}}\label{appendix:A}

Suppose that $\{ c_\lambda \} \in c_{00}(\Lambda)$ is a sequence of finitely many non-zero components. We start by upper bounding the $L^2$-norm of the sum $\sum_{\lambda \in \Lambda} c_\lambda T_\lambda \phi$:
\begin{equation*}
    \begin{split}
        & \left \| \sum_{\lambda \in \Lambda} c_\lambda T_\lambda \phi \right \|^2_{\ltd} = \left \| \ft \sum_{\lambda \in \Lambda} c_\lambda T_\lambda \phi \right \|^2_{\ltd} = \left \| \sum_{\lambda \in \Lambda} c_\lambda M_{-\lambda} \ft \phi \right \|^2_{\ltd} \\
        &= \int_\Rd \left | \sum_{\lambda \in \Lambda} c_\lambda e^{-2\pi i \lambda \cdot t} \ft \phi(t) \right |^2 \, dt = \int_{\Rd} |F(t)|^2|\ft \phi(t)|^2 \, dt
    \end{split}
\end{equation*}
where $F(t) \coloneqq \sum_{\lambda \in \Lambda} c_\lambda e^{-2\pi i \lambda \cdot t}$. The function $F$ is measurable, bounded and $\Lambda^*$-periodic. Since $|F|^2,|\ft \phi|^2 \geq 0$ we therefore obtain
\begin{equation*}
    \begin{split}
        & \int_{\Rd} |F(t)|^2|\ft \phi(t)|^2 \, dt = \sum_{\lambda^* \in \Lambda^*} \int_{\mathcal{P}(\Lambda^*)} |F(t + \lambda^*)|^2 |\ft \phi(t+ \lambda^*)|^2 \, dt \\
        & = \int_{\mathcal{P}(\Lambda^*)} |F(t)|^2 \sum_{\lambda^* \in \Lambda^*} |\ft \phi(t+ \lambda^*)|^2 \, dt = \int_{\mathcal{P}(\Lambda^*)} |F(t)|^2 \mathfrak{p}_{\Lambda^*}[\phi](t) \, dt \\
        & \leq \det(A^{-T}) \| \mathfrak{p}_{\Lambda^*}[\phi] \|_{L^\infty(\mathcal{P}(\Lambda^*))} \sum_{\lambda \in \Lambda} |c_\lambda|^2.
    \end{split}
\end{equation*}
The term $M \coloneqq \det(A^{-T}) \| \mathfrak{p}_{\Lambda^*}^S[u] \|_{L^\infty(\mathcal{P}(\Lambda^*))}$ is a constant independent of $\{ c_\lambda \} \in c_{00}(\Lambda)$. It follows from \cite[Theorem 3, p. 129]{Young} that the system
$
\{ T_\lambda \phi : \lambda \in \Lambda \}
$
is a Bessel sequence in $\ltd$. In particular, the series $\sum_{\lambda \in \Lambda} c_\lambda T_\lambda^S \phi$ converges unconditionally for every $\{ c_\lambda \} \in \ell^2(\Lambda)$ \cite[Corollary 3.2.5]{christensenBook}.

\subsection{Proof of Lemma \ref{lma:elementary_uniqueness}}\label{appendix:B}

Denote by $\lambda^d$ the $d$-dimensional Lebesgue measure in $\Rd$. We prove the statement by induction over the dimension $d$. If $d=1$ then the statement holds true since every set $\mathcal{Y} \subseteq \R$ of positive $1$-dimensional Lebesgue measure, $\lambda^1(\mathcal{Y}) > 0$, contains a limit point, and a limit point is a uniqueness set for holomorphic functions $F : \C \to \C$. This proves the base case. Suppose now that the statement holds in $\C^d$ and let $\mathcal{Y} \subseteq \R^{d+1}$ with $\lambda^{d+1}(\mathcal{Y}) > 0$. Writing $x \in \R^{d+1}$ as $x=(x',t), x' \in \Rd, t \in \R$, and using Fubini's theorem we obtain
\begin{equation}\label{eq:vanishing}
    \lambda^{d+1}(\mathcal{Y}) = \int_{\R^{d+1}} \bone_\mathcal{Y}(x) \, d\lambda^{d+1}(x) = \int_\R\underbrace{ \int_\Rd \bone_\mathcal{Y}(x',t) \, d\lambda^d(x')}_{\coloneqq M(t)} d\lambda^1(t) >0.
\end{equation}
Since $M(t) \geq 0$ for every $t \in \R$, equation \eqref{eq:vanishing} implies that there exists a Lebesgue measurable set $A \subseteq \R$ such that $\lambda^1(A) > 0$ and $M(t)>0$ for every $t \in \R$. It follows by assumption on the function $F$ and by the induction hypothesis that $F(\cdot, t)$ vanishes identically for every $t \in A$. Consequently, the base case shows that $F(x,\cdot)$ vanishes identically for every $x \in \Cd$, thereby proving the statement.

\vspace{0.5cm}

\textbf{Acknowledgement.} The authors appreciate helpful discussions with Irina Shafkulovska and are grateful for the valuable comments made by the reviewers.

\bibliographystyle{acm}
\bibliography{bibfile}

\end{document}